\documentclass[reqno, letterpaper, oneside]{article}
\usepackage{amsmath,amsfonts,amssymb,amsthm}
\usepackage[final]{graphicx}
\usepackage[usenames,dvipsnames]{color}

\usepackage{bbm}
\usepackage{setspace}
\usepackage{geometry}

\newtheorem{theorem}{Theorem}
\newtheorem{lemma}[theorem]{Lemma}
\newtheorem{corollary}[theorem]{Corollary}
\newtheorem{claim}[theorem]{Claim}

\newtheorem{remark}[theorem]{Remark}
\theoremstyle{definition}

\newcommand{\refT}[1]{Theorem~\ref{#1}}
\newcommand{\refC}[1]{Corollary~\ref{#1}}
\newcommand{\refL}[1]{Lemma~\ref{#1}}
\newcommand{\refR}[1]{Remark~\ref{#1}}
\newcommand{\refS}[1]{Section~\ref{#1}}

\newcommand{\refApp}[1]{Appendix~\ref{#1}}

\newcommand{\refCl}[1]{Claim~\ref{#1}}

\newcommand\NN{\mathbb{N}}

\newcommand\RR{{\mathbb R}}

\newcommand\E{\operatorname{\mathbb E{}}}
\renewcommand\P{\operatorname{\mathbb P{}}}

\newcommand\Bin{\operatorname{Bin}}

\newcommand\set[1]{\ensuremath{\{#1\}}}
\newcommand\bigset[1]{\ensuremath{\bigl\{#1\bigr\}}}
\newcommand\Bigset[1]{\ensuremath{\Bigl\{#1\Bigr\}}}

\newcommand\xpar[1]{(#1)}
\newcommand\bigpar[1]{\bigl(#1\bigr)}
\newcommand\Bigpar[1]{\Bigl(#1\Bigr)}
\newcommand\biggpar[1]{\biggl(#1\biggr)}
\newcommand\lrpar[1]{\left(#1\right)}
\newcommand\bigsqpar[1]{\bigl[#1\bigr]}

\newcommand\bigabs[1]{\bigl|#1\bigr|}
\newcommand\Bigabs[1]{\Bigl|#1\Bigr|}

\def\rompar(#1){\textup(#1\textup)}    % usage: \rompar(...)

\newcommand\parfrac[2]{\lrpar{\frac{#1}{#2}}}

\def\xexp(#1){e^{#1}}
\newcommand\ceil[1]{\lceil#1\rceil}
\newcommand\bigceil[1]{\bigl\lceil#1\bigr\rceil}
\newcommand\floor[1]{\lfloor#1\rfloor}
\newcommand\bigfloor[1]{\bigl\lfloor#1\bigr\rfloor}

\newcommand\eqd{\overset{\mathrm{d}}{=}}

\newcommand{\vF}{T}
\newcommand{\vE}{E}
\newcommand{\vO}{O}
\newcommand{\cpi}{\pi}
\newcommand{\csig}{\sigma}
\newcommand{\cPsi}{\Psi}
\newcommand{\ctau}{\tau}
\newcommand{\eps}{\epsilon}
\newcommand{\cq}{q}

\newcommand{\hc}{\hat{c}}
\newcommand{\hp}{\hat{p}}

\newcommand{\chC}{\hat{\cC}}
\newcommand{\cD}{\mathcal{D}}
\newcommand{\cF}{\mathcal{F}}

\newcommand{\cI}{\mathcal{I}}
\newcommand{\cJ}{\mathcal{J}}
\newcommand{\cK}{\mathcal{K}}

\newcommand{\cM}{\mathcal{M}}
\newcommand{\cN}{\mathcal{N}}

\newcommand{\cP}{\mathcal{P}}
\newcommand{\cQ}{\mathcal{Q}}
\newcommand{\cT}{\mathcal{T}}

\newcommand{\fS}{\mathfrak{S}}

\newcommand{\cB}{\mathcal{B}}
\newcommand{\cC}{C}
\newcommand{\cR}{\mathcal{R}}
\newcommand{\cW}{S}
\newcommand{\cX}{X}
\newcommand{\cY}{Y}
\newcommand{\cZ}{Z}
\newcommand{\fX}{\mathcal{X}}

\newcommand{\cG}{\mathfrak{X}}

\renewcommand{\emptyset}{\varnothing} % \text{\O}

\newcommand{\indic}[1]{\mathbbm{1}_{\{{#1}\}}}

\let\OLDthebibliography\thebibliography
\renewcommand\thebibliography[1]{
  \OLDthebibliography{#1}
  \setlength{\parskip}{0pt}
  \setlength{\itemsep}{0pt plus 0.3ex}
}

\title{Packing nearly optimal Ramsey~$R(3,t)$ graphs}
\author{He Guo and Lutz Warnke%
\thanks{School of Mathematics, Georgia Institute of Technology, Atlanta GA~30332, USA. 
E-mail: {\tt he.guo@gatech.edu, warnke@math.gatech.edu}.
Research partially supported by NSF Grant DMS-1703516.}}
\date{November 14, 2017}

\begin{document}
\maketitle

\begin{abstract}
In 1995 Kim famously proved the Ramsey bound~$R(3,t) \ge c t^2/\log t$ by constructing 
an $n$-vertex graph that is triangle-free and has independence number at most~$C \sqrt{n \log n}$. 
We extend this celebrated result, which is best possible up to the value of the constants, 
by approximately decomposing the complete graph~$K_n$ into a packing of such nearly optimal Ramsey~$R(3,t)$ graphs. 

More precisely, for any $\epsilon>0$ we find an edge-disjoint collection $(G_i)_i$ of $n$-vertex graphs $G_i \subseteq K_n$ such that 
(a)~each $G_i$ is triangle-free and has independence number at most $C_\epsilon \sqrt{n \log n}$, 
and (b)~the union of all the $G_i$ contains at least $(1-\epsilon)\binom{n}{2}$ edges. 
Our algorithmic proof proceeds by sequentially choosing the graphs~$G_i$ 
via a semi-random (i.e., R\"{o}dl nibble type) variation of the triangle-free process.  

As an application, we prove a conjecture in Ramsey theory by Fox, Grinshpun, Liebenau, Person, and Szab\'{o} 
(concerning a Ramsey-type parameter introduced by Burr, Erd\H{o}s, Lov\'asz in 1976). 
Namely, denoting by~$s_r(H)$ the smallest minimum degree of $r$-Ramsey minimal graphs for~$H$, 
we close the existing logarithmic gap for~$H=K_3$ and establish that~$s_r(K_3) = \Theta(r^2 \log r)$. 
\end{abstract}

\section{Introduction}
The 1947 paper of Erd\H{o}s~\cite{Erdos1947} on the diagonal Ramsey number~$R(t,t)$ is often considered the start of the probabilistic method,
where~$R(s,t)$ is defined as the smallest integer $n \in \NN$ such that every red-blue 
colouring of the edges of the complete $n$-vertex graph~$K_n$ contains either a red~$K_s$ or a blue~$K_t$. 
In general, the estimation of~$R(s,t)$ and other Ramsey-type parameters is known to be notoriously difficult.

One of the celebrated results in Ramsey theory is~$R(3,t)=\Theta(t^2/\log t)$, 
and this special case has 
repeatedly served as a testbed for the development of new tools and techniques in probabilistic combinatorics. 
Indeed, complementing the basic bound $R(3,t) = O(t^2)$ of Erd\H{o}s and Szekeres~\cite{ES1935},  
in~1961 Erd\H{o}s~\cite{Erdos1961} used a sophisticated random greedy alteration argument to prove $R(3,t) = \Omega(t^2/(\log t)^2)$. 
This lower bound was subsequently reproved (or only slightly improved) 
using the Lov\'asz Local Lemma~\cite{Spencer1977}, 
a~basic analysis of the triangle-free process\footnote{The triangle-free process (proposed by Bollob{\'a}s and Erd{\H o}s) proceeds as follows: 
starting with an empty $n$-vertex graph, in each step a single edge is added, chosen uniformly at random from all non-edges which do not create a triangle.\label{fn:tfprocess}}~\cite{ESW1995}, 
large deviation inequalities~\cite{Krivelevich1995}, 
and differential equations~\cite{K3maximal}. 
Furthermore, in~1980 Ajtai, Koml\'os, and Szemer\'edi~\cite{AKS1980,AKS1981} invented the 
influential semi-random method (nowadays also called R\"{o}dl nibble approach) 
to prove the upper bound~$R(3,t) = O(t^2/\log t)$. 
But it was not until~1995, when Kim~\cite{Kim} famously proved the matching lower bound~$R(3,t) = \Omega(t^2/\log t)$ 
by analyzing a semi-random variation of the triangle-free process\footnote{Kim's semi-random variation proceeds similar to the triangle-free process, but intuitively adds a large number of carefully chosen random-like edges in each step (instead of just a single edge); 
see \refS{sec:nibble} for more details.\label{fn:semirandom}} 
(combining several of the aforementioned ideas with martingales concentration);  
for this major breakthrough he also received the Fulkerson Prize in~1997. 
But the story does not end here: advancing the differential equation method, 
in~2008 Bohman~\cite{Bohman} reproved~$R(3,t) = \Omega(t^2/\log t)$ 
by analyzing the triangle-free process itself 
(and his analysis was recently further improved in~\cite{BK2013,FPGM2013}).

In this paper we refine the powerful techniques developed for~$R(3,t)=\Theta(t^2/\log t)$ 
to determine the order of magnitude of %the special case of 
another Ramsey-type parameter introduced in~1976 by Burr, Erd\H{o}s, and Lov\'asz~\cite{BEL1976}, 
proving a conjecture of Fox, Grinshpun, Liebenau, Person, and Szab\'{o}~\cite{Fox} 
(in particular, analogous to Kim's $R(3,t)$-result, we again remove the last redundant logarithmic factor from existing bounds).

\subsection{Main result: packing of nearly optimal Ramsey~$R(3,t)$ graphs}\label{sec:intro:main}
Kim and Bohman both proved the Ramsey bound~$R(3,t) = \Omega(t^2/\log t)$ 
by showing the existence of a triangle-free graph $G \subseteq K_n$ on $n$~vertices with independence number~$\alpha(G) = O(\sqrt{n \log n})$, 
which is best possible up to the value of the implicit constants. 
Our first theorem naturally extends their celebrated results, by approximately decomposing the 
complete graph~$K_n$ into a packing of such nearly optimal Ramsey~$R(3,t)$~graphs.  
\begin{theorem}\label{thm:main}
For any $\eps >0$ there exist $n_0, C,D>0$ such that, for all $n \ge n_0$,  
there is an edge-disjoint collection $(G_i)_{i \in \cI}$ of 
$|\cI| = \ceil{D \sqrt{n/\log n}}$ triangle-free graphs $G_i \subseteq K_n$ 
on~$n$ vertices with 
$\max_{i \in \cI}\alpha(G_i)\le C\sqrt{n\log n}$ and 
$\sum_{i \in \cI} e(G_i)\ge(1-\eps)\binom{n}{2}$. 
\end{theorem} 
Our algorithmic proof proceeds by sequentially choosing the $|\cI| = \Theta(\sqrt{n/\log n})$ 
edge-disjoint triangle-free subgraphs $G_i \subseteq K_n \setminus \bigcup_{0 \le j < i}G_j$ 
with~$\alpha(G_i) = O(\sqrt{n \log n})$ via a semi-random variation of the triangle-free process 
akin~Kim~\cite{Kim} (see Sections~\ref{sec:intro:main:tool} and~\ref{sec:nibble} for the details). 
In particular, we do not only show existence of the~$(G_i)_{i \in \cI}$, 
but also obtain a polynomial-time randomized algorithm which constructs these subgraphs. % efficiently.  

\refT{thm:main} improves a construction of Fox~et.al.~\cite[Lemma~4.2]{Fox}, 
who used the basic Lov\'asz Local Lemma based $R(3,t)$-approach to sequentially choose 
$\Theta(\sqrt{n}/\log n)$ edge-disjoint triangle-free subgraphs with~$\alpha(G_i) = O(\sqrt{n} \log n)$. 
It is natural to suspect that applying a more sophisticated $R(3,t)$-approach in each 
iteration ought to give an improved packing (with smaller independence number than the LLL~approach), 
and here the usage of the triangle-free process was 
proposed by Fox~et.al.~\cite[Section~5]{Fox} as early as~2013~\cite{Liebenau,Person}. % , if not earlier 
One conceptual difficulty of this approach is to control various error terms 
over many iterations of the triangle-free process 
(so that these always stay small enough to carry out the next iteration), 
which in turn is the main technical reason why for \refT{thm:main} we instead iterate a semi-random variation.

It would be interesting to know if \refT{thm:main} also holds with $\eps=0$, 
i.e., if one can completely decompose~$K_n$ into nearly optimal~$R(3,t)$ graphs. 
Perhaps rashly, we conjecture that this is indeed possible 
%, and one natural approach is to try constructing the edge-sets of all the~$G_i$ in parallel rather than sequentially  
(it might be insightful to first prove a variant of \refT{thm:main} where the constant~$C$ does \emph{not} depend on~$\eps$).

\subsection{Application in Ramsey theory: $s_r(K_3)$ has order of magnitude $r^2 \log r$}\label{sec:intro:apl}
Turning to our main application, 
we say that a graph~$G$ is \emph{$r$-Ramsey for $H$}, denoted by~$G\rightarrow (H)_r$, 
if any $r$-colouring of the edges of $G$ contains a monochromatic copy of~$H$. 
Most fundamental questions and results in Ramsey theory can be formulated in terms 
various parameters of the class 
\begin{equation*}
\cM_r(H) := \bigset{G: \: \text{$G\rightarrow (H)_r$ and $G' \nrightarrow (H)_r$ for all $G' \subsetneq G$}} %,
\end{equation*}
of graphs which are \emph{$r$-Ramsey minimal for $H$}. 
For example, Ramsey's theorem~\cite{Ramsey} states that~$|\cM_r(H)|>0$ for all graphs~$H$, 
which for cliques was strengthened to~$|\cM_r(K_t)|=\infty$ by R\"odl and Siggers~\cite{RS2008}. %$|\cM_2(K_t)|=\infty$ by Burr, Erd\H{o}s, and Lov\'asz~\cite{BEL1976}
Furthermore, the archetypal problem of estimating various Ramsey-type parameters 
also corresponds to the study of certain extremal parameters of $\cM_r(H)$, since, e.g., 
$R(t)=R(t,t):=\min_{G \in \cM_2(K_t)} v(G)$ is the famous diagonal Ramsey number~\cite{ES1935,Erdos1947,ConlonFoxSudakov2015}, 
$R_r(t) = R(t, \ldots, t):=\min_{G \in \cM_r(K_t)} v(G)$ is the $r$-coloured Ramsey number~\cite{ConlonFoxSudakov2015}, %of~$K_t$ 
and $\hat R_r(H):=\min_{G \in \cM_r(H)} e(G)$ is the widely-studied $r$-size-Ramsey number of~$H$ (see, e.g.,~\cite{EFRS1978,Beck1983,RS2000,ConlonFoxSudakov2015}).

In~1976 Burr, Erd\H{o}s, and Lov\'asz~\cite{BEL1976} initiated the systematic study of other extremal parameters of~$\cM_r(H)$, 
including the smallest minimum degree of all $r$-Ramsey minimal graphs for~$H$, denoted by 
\begin{equation*}
s_r(H) := \min_{G \in \cM_r(H)} \delta(G) .
\end{equation*}
As usual, the clique-case $H=K_t$ is of particular interest, 
where $r (t-2) < s_r(K_t) < R_r(t)$ is easy to see~(cf.~\cite{Fox2007,SZZ2010}). 
Perhaps surprisingly, for $r=2$~colours Burr~et.al.~\cite{BEL1976} were able to prove $s_2(K_t)=(t-1)^2$, 
showing that the simple exponential upper bound~$R_2(t)=R(t) = 2^{\Theta(t)}$ is far from the truth. 
For~$r \ge 2$~colours the behaviour of $s_r(K_t)$ was recently investigated in detail by Fox~et.al.~\cite{Fox}: 
they proved super-quadratic bounds of form $s_r(K_t)=r^2\cdot \text{polylog} \ r$ for fixed $t \ge 3$, 
and also determined $s_r(K_3)$ up to a logarithmic factor (by sharpening their general estimates). 
In particular, they showed $c r^2 \log r \le s_r(K_3) \le C r^2 (\log r)^2$, 
and conjectured that their lower bound gives the correct order of magnitude, see~\cite[Conjecture~5.4]{Fox}. 
%(this, e.g., also appears on the open problem list of the 2015 AIM `Graph Ramsey theory' workshop). 

%
Our second theorem proves the aforementioned conjecture of 
Fox, Grinshpun, Liebenau, Person, and Szab\'{o} for~$s_r(K_3)$, 
i.e., we close the logarithmic gap and establish $s_r(K_3) = \Theta(r^2 \log r)$. 
\begin{theorem}\label{thm:srK3}
There exists $C>0$ such that $s_r(K_3) \le C r^2 \log r$ for all $r \ge 2$. 
\end{theorem}
\begin{corollary}
We have $s_r(K_3) = \Theta(r^2 \log r)$ for $r \ge 2$. 
\end{corollary}
Using a reformulation of~$s_r(K_3)$ from~\cite{Fox}, \refT{thm:srK3} follows easily from our main packing result. 
Indeed, applying \refT{thm:main} with~$\eps=1/2$, say, 
it is routine to see that there is a constant~$A>0$ such that the following holds for each~$r \ge 2$:
there exists a collection of edge-disjoint triangle-free graphs~$G_1, \ldots, G_r \subseteq K_{N_r}$ 
on~$N_r:= \floor{A r^2 \log r}$ vertices with independence number~$\alpha(G_i) < N_r/r$ 
(as~$N_r \ge n_0$, $D \sqrt{N_r/\log N_r} \ge r$ and $C \sqrt{N_r \log N_r} < N_r/r$ all hold for~$A=A(n_0,C,D)$ large enough). 
By~Theorem~1.5 and Lemma~4.1 in~\cite{Fox} (with~$n=N_r$ and~$k=2$) this immediately implies $s_r(K_3) \le N_r$, 
establishing \refT{thm:srK3}. % (with~$C=A$). 

Note that the above deduction of \refT{thm:srK3} did \emph{not} use $\sum_{i \in \cI} e(G_i)\ge(1-\eps)\binom{n}{2}$, 
i.e., that the nearly optimal~$R(3,t)$ graphs $(G_i)_{i \in \cI}$ approximately decompose the edge-set of~$K_n$. 
It would be interesting to find applications (e.g., in Ramsey theory or extremal combinatorics) 
where this natural packing property is useful.

\subsection{Main tool: pseudo-random triangle-free subgraphs}\label{sec:intro:main:tool}
The $R(3,t)$-proofs of Kim and Bohman both in fact construct a 
triangle-free graph~$G \subseteq K_n$ with pseudo-random properties 
%that are strong enough to bound the independence number from above 
(see also~\cite{K3maximal,Wolfovitz2011,BK2013,FPGM2013}). 
Our third theorem extends their intriguing results to host graphs~$H \subseteq K_n$ which are far from complete, 
by showing that one can again construct a triangle-free subgraph~$G \subseteq H$ with pseudo-random properties.  
Here the crux is that \refT{thmiteration} holds under very weak assumptions,\footnote{Note that \refT{thmiteration} does \emph{not} require the host graph~$H$ to be approximately degree or codegree regular. 
Furthermore, even if~$G \subseteq H$ was a random subgraph with edge-probability~$\rho$, 
then by standard calculations we would only expect the edge-estimate~\eqref{eq:thmiteration:edge} 
to hold for vertex-sets $A,B\subseteq V(H)$ where the number of edges~$e_H(A,B)$ is reasonably large 
(see \refR{rem:prop:const} and Footnote~\ref{fn:prop:const} on page~\pageref{fn:prop:const}, which also indicate that 
the constant~$C$ in \refT{thmiteration} has the correct dependence on $\gamma,\delta,\beta$).}  
and that~$G$ resembles a random subgraph of~$H$ with edge-probability~$\rho=\Theta(\sqrt{(\log n)/n})$. % (see~\eqref{eq:thmiteration:edge} below). 
\begin{theorem}\label{thmiteration}
There exist $\beta_0,D_0>0$ such that, 
for all $\gamma,\delta \in (0,1]$, $\beta\in (0,\beta_0)$ and $C\ge D_0/(\delta^2\sqrt{\beta}\gamma)$, 
the following holds for all $n\ge n_0(\gamma,\delta,\beta,C)$, with $\rho := \sqrt{\beta(\log n)/n}$. 
For any $n$-vertex graph~$H$, there exists a triangle-free subgraph~$G \subseteq H$ on the same vertex-set such that 
\begin{equation}\label{eq:thmiteration:edge}
e_G(A,B)=(1\pm\delta)\rho e_H(A,B) 
\end{equation}
for all disjoint vertex-sets $A,B\subseteq V(H)$ with $|A|=|B|=\lceil C\sqrt{n\log n}\rceil$ and $e_H(A,B) \ge \gamma |A||B|$. 
\end{theorem}

Our proof uses a semi-random variant of the triangle-free process to construct~$G \subseteq H$, 
extending and simplifying Kim's approach for the complete case~$H = K_n$ 
(see Sections~\ref{sec:nibble}--\ref{sec:proof} and \refT{prop} for the details). 
In particular, besides handling the difficulties arising due to incomplete host graphs~$H \subseteq K_n$ 
(by, e.g., exploiting a `stabilization mechanism' to keep various parameters under control), 
the major technical difference lies in the way we analyze the properties of all large vertex-sets 
(by, e.g., focusing on bipartite subgraphs, applying a concentration inequality of Warnke~\cite{APUT}, 
%looking at the entire semi-random construction as one random process,  
and showing concentration in~\eqref{eq:thmiteration:edge} instead of just~$e_G(A,B) \ge 1$). 
Together with some streamlining of Kim's arguments 
(by, e.g., using fewer variables, applying convenient bounded differences inequalities, and some changes to the semi-random construction), 
this leads to a shorter and hopefully more accessible proof even in the complete case~$H=K_n$. 
As a by-product, we also obtain a randomized polynomial-time algorithm which 
constructs~$G \subseteq H$ efficiently (see \refR{rem:prop:poly}).

\refT{thmiteration} will be the main tool for establishing our main packing result \refT{thm:main}.  
Let us briefly sketch the  argument (deferring the details to \refS{sec:iterate}). 
The idea is to sequentially choose the triangle-free subgraphs $G_{i} \subseteq H_i := K_n \setminus \bigcup_{0 \le j < i}G_j$ via \refT{thmiteration} with $\delta \in (0,1)$, 
using the pseudo-random edge-estimate~\eqref{eq:thmiteration:edge} to inductively control the number of remaining edges (between large sets) in~$H_i$ as 
\begin{equation}\label{eq:iteration:edge}
e_{H_i}(A,B) =  (1 - (1 \pm \delta)\rho)^i \cdot |A| |B| \qquad \text{for all disjoint $A,B\subseteq V(H)$ of size~$s:=\ceil{C \sqrt{n \log n}}$},
\end{equation}
stopping when the right hand side of~\eqref{eq:iteration:edge} drops below~$\eps |A||B|$ after $I=\Theta(\log(1/\eps)/\rho) = \Theta(\sqrt{n/\log n})$ steps. 
A double counting argument will then show that the leftover  graph~$H_I$ contains at most~$\eps \binom{n}{2}$ edges, 
so that~$\sum_{0 \le i < I}e(G_i) = e(K_n \setminus H_I) \ge (1-\eps)\binom{n}{2}$. 
Furthermore, $e_{G_i}(A,B) = (1 \pm \delta) \rho e_{H_i}(A,B)>0$ implies $\alpha(G_i) < 2s = O(\sqrt{n \log n})$, 
completing this rough proof sketch of \refT{thm:main} (assuming \refT{thmiteration}).

We believe that variants of Theorems~\ref{thm:main} and~\ref{thmiteration} also hold for many other forbidden graphs 
(using semi-random variants of the $H$-free process~\cite{OT2001,BK2010,WarnkeK42014,WarnkeCl2014,Picollelli2014});  
we hope to return to this topic in a future work.

\subsection{Organization of the paper}\label{sec:org}
The remainder of this paper is organized as follows.
In \refS{sec:iterate} we use \refT{thmiteration} to state and prove some extensions of our main packing result \refT{thm:main}.  
In \refS{sec:nibble} we introduce a semi-random  variation of the triangle-free process 
and state our main result for this R\"{o}dl nibble type construction 
(that implies our main tool \refT{thmiteration}, see \refS{sec:mainnibble}), 
which is then subsequently proved in \refS{sec:proof}.

\subsection{Further results}\label{sec:iterate}

Our methods allow us to extend \refT{thm:main} to $R(3,t)$-packings of graphs which are far from complete. 
Our fourth theorem shows that if~$H \subseteq K_n$ only satisfies certain uniformity conditions on its edge distribution 
(that resemble a weak form of pseudo-randomness, see~\eqref{eq:maintheorem:ass} below), 
then we can still  approximately decompose~$H$ into a packing of nearly optimal Ramsey~$R(3,t)$~graphs 
(again by an efficient randomized algorithm). 
\begin{theorem}\label{maintheorem}
For all $\eps,\xi,C_0 >0$ there exist $n_0,C_1,D>0$ such the following holds for all~$n \ge n_0$.
If $H$ is an $n$-vertex graph satisfying 
\begin{equation}\label{eq:maintheorem:ass}
\min_{\substack{\text{disjoint $A,B\subseteq V(H)$:}\\
	|A|=|B| = \ceil{C_0\sqrt{n\log n}}}}\frac{e_H(A,B)}{|A||B|}\ge\xi,
\end{equation}
then there is an edge-disjoint collection $(G_i)_{i \in \cI}$ of 
$|\cI| = \ceil{D \sqrt{n/\log n}}$ triangle-free subgraphs $G_i \subseteq H$ 
with~$V(G_i) = V(H)$, 
$\max_{i \in \cI}\alpha(G_i)\le C_1\sqrt{n\log n}$ and 
$\sum_{i \in \cI} e(G_i)\ge(1-\eps) e(H)$.
\end{theorem} 
Note that the case $H=K_n$ and $\xi=C_0=1$ implies \refT{thm:main}. 
Furthermore, the case $H=G_{n,p}$, $\xi = p/2$ and~$C_0=1$  routinely implies 
the following sparse analogue of \refT{thm:main} for binomial random graphs~$G_{n,p}$. 
\begin{corollary}\label{maincorollary}
For any $p \in (0,1]$ and $\eps >0$ there exist $C,D>0$ such that, with probability at least $1-o(1)$, 
the following event holds: 
there exists an edge-disjoint collection $(G_i)_{i \in \cI}$ of 
$|\cI| = \ceil{D \sqrt{n/\log n}}$ triangle-free graphs $G_i \subseteq G_{n,p}$ 
on $n$~vertices with   
$\max_{i \in \cI}\alpha(G_i)\le C\sqrt{n\log n}$ and 
$\sum_{i \in \cI} e(G_i) = (1 \pm \eps)p\binom{n}{2}$. 
\end{corollary} 
We conjecture that \refC{maincorollary} (with $|\cI| = \ceil{D p\sqrt{n/\log n}}$ and constants~$C,D>0$ depending only on~$\eps$) holds for 
much sparser random graphs~$G_{n,p}$ with edge-probabilities of form~$p=p(n) \ge n^{-1/2+o(1)}$, say.\footnote{The range of~$p=p(n)$ in this conjecture is essentially best possible, since it is well-known that typically $\alpha(G_{n,p}) \gg \sqrt{n \log n}$ for $p \ll \sqrt{(\log n)/n}$. 
Furthermore, although we have not checked all details, it seems that our proofs can be modified 
to verify the conjecture for~$p \ge n^{-\delta}$, where~$\delta>0$ is some small constant;  
so the main question is whether $p \ge n^{-1/2+o(1)}$ suffices.}

We conclude the introduction with the short proof of \refT{maintheorem}, 
which proceeds by sequentially choosing the graphs~$G_{i} \subseteq H \setminus \bigcup_{0 \le j < i}G_j$ via \refT{thmiteration} 
(generalizing the argument sketched in \refS{sec:intro:main:tool}). 
The reader mainly interested in the proof of \refT{thmiteration}
may perhaps wish to skip straight to Section~\ref{sec:nibble}. 
\begin{proof}[Proof of \refT{maintheorem} (assuming \refT{thmiteration})] 
We may assume~$\eps < 1$ (as decreasing~$\eps$ gives a stronger conclusion). % (since decreasing $\eps>0$ gives a stronger conclusion). 
For concreteness, set $\delta := 1/4$, $\gamma:=\eps^2\xi$, $\beta := \beta_0/2$ 
and $C:=\max\{C_0, \: D_0/(\delta^2\sqrt{\beta}\gamma)\}$, 
where $\beta_0,D_0$ are defined as in Theorem~\ref{thmiteration}. 
Let $C_1 := 3C$, $s:= \ceil{C\sqrt{n\log n}}$, $\rho := \sqrt{\beta(\log n)/n}$, and $I:=\lceil\log(1/\eps)/(\rho(1-\delta))\rceil$.

Define $H_0 := H$. 
Let $\fS$ denote the set of all pairs $(A,B)$ of disjoint vertex-sets $A,B \subseteq V(H)$ with $|A|=|B|=s$.
Combining a `handshaking lemma' like double counting argument with the assumed lower bound~\eqref{eq:maintheorem:ass}, 
writing $t := \ceil{C_0\sqrt{n\log n}}$ it follows that 
\begin{equation}\label{eq:handshake}
\frac{e_{H_0}(A,B)}{|A||B|}
=\frac{\sum_{\tilde{A}\subseteq A,\tilde{B}\subseteq B: \: |\tilde{A}|=|\tilde{B}|=t}e_H(\tilde{A},\tilde{B})}{s^2 \cdot \binom{s-1}{t-1}\binom{s-1}{t-1}}
\ge \frac{\binom{s}{t}\binom{s}{t} \cdot \xi t^2}{s^2 \binom{s-1}{t-1}\binom{s-1}{t-1}}
=\xi 
\qquad \text{for all $(A,B) \in \fS$}.
\end{equation}

The plan is to sequentially choose the graphs $(G_i)_{0 \le i < I}$ with $G_i \subseteq H_i$
such that, setting $H_{i+1}:=H_i \setminus G_i$ (which ensures that all the~$G_i$ are edge-disjoint), 
for all $0 \le i \le I$ we inductively have 
\begin{equation}\label{eq:invariant}
\frac{e_{H_i}(A,B)}{e_{H_0}(A,B)}\in \Bigl[\bigpar{1-(1+\delta)\rho}^i,\: \bigpar{1-(1-\delta)\rho}^i \Bigr] \qquad \text{for all $(A,B) \in \fS$}.
\end{equation}
Turning to the details, note that inequality~\eqref{eq:invariant} holds trivially for~$i=0$. 
Given~$H_i$ with $0 \le i \le I-1$ satisfying~\eqref{eq:invariant}, 
by combining the definition of~$I$ with $(1+2\delta)/(1-\delta) = 2$ and~\eqref{eq:handshake}  
it follows for $n \ge n_0(\beta)$ that, say, 
\begin{equation}\label{eq:lowerbound}
\frac{e_{H_i}(A,B)}{|A||B|} \ge e^{-(1+2\delta) \rho (I-1)} \cdot \frac{e_{H_0}(A,B)}{|A||B|} \ge \eps^2 \cdot \xi = \gamma  \qquad \text{for all $(A,B) \in \fS$}.
\end{equation}
Using \refT{thmiteration}, for $n\ge n_0(\eps,\xi,\delta,\beta,C)$ we can thus find 
a triangle-free subgraph~$G_i \subseteq H_i$ with $e_{G_i}(A,B)=(1\pm\delta)\rho e_{H_i}(A,B) > 0$ for all $(A,B) \in \fS$. 
Hence $\alpha(G_i) < 2s \le 3C\sqrt{n\log n}$, say. 
Furthermore, noting $e_{H_{i+1}}(A,B)=e_{H_i}(A,B)-e_{G_i}(A,B)$, it is immediate that 
$H_{i+1}=H_i \setminus G_i$ maintains~\eqref{eq:invariant}.

Finally, for the number of edges of $\bigcup_{0 \le i < I}G_i = H_0 \setminus H_I$,  
by~\eqref{eq:invariant} and definition of~$I$ it follows that 
\begin{equation}\label{eq:lowerbound2}
e_{H_0\setminus H_I}(A,B) %= e_{H_0}(A,B) - e_{H_I}(A,B)
\ge \bigpar{1-e^{-(1-\delta)\rho I}} \cdot e_{H_0}(A,B) \ge (1-\eps) e_{H_0}(A,B) 
\qquad \text{for all $(A,B) \in \fS$}.
\end{equation}
Using a double counting argument similar to~\eqref{eq:handshake}, 
in view of~\eqref{eq:lowerbound2} and $H_0=H$ we infer 
\[
e(H_0\setminus H_I)
=\frac{\sum_{(A,B) \in \fS} e_{H_0\setminus H_I}(A,B)}{2\binom{n-2}{s-1}\binom{n-2-(s-1)}{s-1}}
\ge (1-\eps) \cdot \frac{\sum_{(A,B) \in \fS} e_{H}(A,B)}{2\binom{n-2}{s-1}\binom{n-2-(s-1)}{s-1}} = (1-\eps) e(H) ,
\] 
completing the proof of $\sum_{0 \le i < I} e(G_i) = e(H_0\setminus H_I) \ge(1-\eps)e(H)$. 
\end{proof}

\section{The nibble: semi-random triangle-free process}\label{sec:nibble}
The remainder of this paper is devoted to the proof of our main tool \refT{thmiteration}. 
Given an $n$-vertex graph~$H$ with vertex-set $V=V(H)$ and edge-set~$E(H)$, 
inspired by Kim~\cite{Kim} our strategy is to incrementally construct the triangle-free edge-set of $G \subseteq H$ 
using a semi-random variation of the triangle-free process (adding large chunks of random-like edges in each step; see also Footnotes~\ref{fn:tfprocess}--\ref{fn:semirandom} on page~\ref{fn:tfprocess}). 
One key difference to~\cite{Kim,Bohman} is that our approach only uses edges from the host graph~$H$ (and not the complete graph~$K_n$). 
In particular, deferring the details to \refS{sec:nibbledetails}, the rough plan of our R\"{o}dl nibble type construction 
is to step-by-step build up a `random' set of edges~$\vE_i \subseteq E(H)$ and a triangle-free subset~$\vF_i \subseteq \vE_i$;  
we also keep track of a set
\begin{equation}\label{def:Oi:subset}
\vO_i \subseteq \set{ e \in E(H) \setminus \vE_i: \:  \text{$e$ does not form a triangle with any two edges of $\vE_i$}}
\end{equation}
of `open' edges that can still be added. 
The idea of each step is to choose a small number of random edges~$\Gamma_{i+1} \subseteq \vO_i$ 
so that only a few new triangles are created in~$\vE_{i+1} = \vE_i \cup \Gamma_{i+1}$. 
This allows us to find an edge-subset $\Gamma'_{i+1} \subseteq \Gamma_{i+1}$, with $|\Gamma'_{i+1}| \approx |\Gamma_{i+1}|$,    
such that $\vF_{i+1} = \vF_i \cup \Gamma'_{i+1}$ remains triangle-free.\footnote{For the construction of~$\vF_{i+1}$ it might seem overly complicated 
to define~$\vO_i$ with respect to~$\vE_i$ (and not~$\vF_{i}$). 
However, this slightly wasteful definition actually simplifies the analysis: 
e.g., for the purpose of tracking various auxiliary variables, 
it intuitively is easier to understand the effect of adding the random edges~$\Gamma_{i+1}$ 
(rather than some subset $\Gamma'_{i+1} \subseteq \Gamma_{i+1}$).
Using an inclusion in~\eqref{def:Oi:subset} might also seem overly complicated, 
but it again simplifies the analysis: by removing some extra edges it actually 
becomes easier to prove concentration (see the `stabilization mechanism' discussion 
around~\eqref{def:pe:cons} and \refL{lemma1}).} 
After 
\begin{align}
\label{def:I}
I &:= \bigceil{n^\beta}
\end{align}
such alteration-method based steps, we eventually obtain a triangle-free graph~$G=(V,\vF_I) \subseteq H$, 
which intuitively ought to be `random enough' to resemble 
(many features of) a random subgraph of~$H$.

\subsection{Details of the nibble construction}\label{sec:nibbledetails}
Turning to the details of the nibble construction, 
consistent with~\eqref{def:Oi:subset} we start with 
\begin{gather}\label{def:start}
\vO_0  := E(H)
 \quad \text{ and } \quad 
\vE_0 := \vF_0 := \Gamma_0 := \emptyset .
\end{gather}
In step $i+1 \ge 1$ we then set 
\begin{align}\label{def:Ei}
\vE_{i+1} &:= \vE_i\cup\Gamma_{i+1} ,
\end{align}
where each edge $e \in \vO_i$ is included in $\Gamma_{i+1}$, independently, with probability 
\begin{align}\label{def:p}
p& :=\csig/\sqrt{n} . 
\end{align}
(The definition of the deterministic parameter $\csig \ll 1$ is deferred to~\eqref{def:sigma} in \refS{sec:def}.)  
Note that~$\vF_i\cup\Gamma_{i+1}$ is not necessarily triangle-free, 
since two or three edges of a triangle could enter via~$\Gamma_{i+1} \subseteq \vO_i$ 
(one edge is not enough by~\eqref{def:Oi:subset} and $\vF_i \subseteq \vE_i$), 
i.e., via the following set of `bad' pairs and triples of $\Gamma_{i+1}$--edges: 
\begin{align}
\cB_{i+1} &:= \bigset{\set{wu,wv}\subseteq \Gamma_{i+1}: \: uv\in\vF_i, \: |\set{u,v,w}|=3 } \: \cup \: \bigset{\set{uv,vw,wu}\subseteq \Gamma_{i+1}: \: |\set{u,v,w}|=3 },
\end{align}
where we write~$xy=\{x,y\}$ for brevity. To avoid triangles in~$T_{i+1}$ by alteration, we thus 
take~$\cD_{i+1}$ to be a maximal collection of pairwise edge-disjoint elements of~$\cB_{i+1}$ 
(say the first one in lexicographic order to resolve ties; any other deterministic choice also works, see \refR{rem:polystep} and \refS{sec:P5}), 
and then set\footnote{The standard alteration approach of removing one edge from each element of~$\cB_{i+1}$ seems harder to analyze: 
e.g., removing the edges of a maximal edge-disjoint collection~$\cD_{i+1} \subseteq \cB_{i+1}$ greatly facilitates the technical calculations in \refS{sec:P5}.}
\begin{align}\label{def:Ti}
\vF_{i+1} &:= \vF_{i}\cup\bigpar{ \Gamma_{i+1} \setminus E(\cD_{i+1})} ,
\end{align}
where we write $E(\cD_{i+1}) := \bigcup_{\alpha\in\cD_{i+1}} \alpha$ for the set of edges in the pairs and triples of~$\cD_{i+1}$.  
Note that~$\vF_{i+1}$ is indeed triangle-free by maximality of~$\cD_{i+1} \subseteq \cB_{i+1}$. 
Turning to the open edge-set~$\vO_{i+1} \subseteq \vO_i \setminus \Gamma_{i+1}$, by~\eqref{def:Oi:subset} 
the set $\cC^1_{i+1} \cup \cC^2_{i+1}$ of newly `closed' edges (that form a triangle with some two edges of $E_{i+1}$) is given~by 
\begin{align}
\cC^1_{i+1} &:= \set{f \in \vO_i : \: \cY_f(i) \cap \Gamma_{i+1} \neq \emptyset} , \\
\label{def:Yuv}
\cY_{uv}(i) &:= \set{uw\in \vO_i: \:  vw\in\vE_i} \: \cup \: \set{vw\in\vO_i : \: uw\in \vE_i},\\
%= \set{uv\in\vO_i : \: \text{there is $w$ s.t. } |\{uw,vw\} \cap \Gamma_{i+1}|=|\{uw,vw\} \cap \vE_{i+1}|=1} . \\
\cC^2_{i+1} &:= \set{uv\in\vO_i : \: \text{there is $w$ s.t. } uw\in\Gamma_{i+1}, \: vw\in\Gamma_{i+1}}.
\end{align}
Mimicking a technical idea of Alon, Kim and Spencer~\cite{AlonKimSpencer1995}, 
we intuitively increase the set of closed edges (via the random set~$S_{i+1}$ below) in order to 
add a `stabilization mechanism' to our construction,\footnote{Kim uses a different stabilization mechanism in~\cite[Section~5.1]{Kim}: 
instead of introducing the random sets~$S_{j}$, 
he deterministically modifies the underlying graphs in each step 
(by temporarily adding some extra edges and vertices), 
mimicking an earlier `regularization' idea from~\cite{Kahn1993}. 
We find our randomized approach more elegant, 
and easier to implement algorithmically.\label{fn:diff}}
and define 
\begin{align}
\label{def:Ci}
\cC_{i+1} &:= \cC^1_{i+1} \cup S_{i+1}, \\
\label{def:Oi}
\vO_{i+1} &:= \vO_i \setminus \bigpar{\Gamma_{i+1}\cup\cC_{i+1} \cup \cC^2_{i+1}} ,
\end{align}
where each edge $e \in \vO_i$ is included in $S_{i+1}$, independently, with `stabilization' probability 
\begin{align}\label{def:pe}
\hp_{e,i} :=1-(1-p)^{\max\set{2\cq_i(\cpi_i+\sqrt{\csig})\sqrt{n}-|\cY_e(i)|, \: 0}}.
\end{align}
(The definition of the deterministic parameters $\cq_i,\cpi_i$ is deferred to~\eqref{def:qi}--\eqref{def:pii} in \refS{sec:def}.) 
Roughly put, the main point of the technical definitions of $S_{i+1}$ and $\hp_{e,i}$ will be that all the conditional probabilities 
\begin{align}\label{def:pe:cons}
\P(e\not\in\cC_{i+1} \mid O_i,E_i) 
%= \P(e \not\in \cC^1_{i+1} \mid (O_i,E_i)) \cdot \P(e \not\in S_{i+1} \mid (O_i,E_i)) =
= \P(e \not\in \cC^1_{i+1} \mid O_i,E_i) \cdot (1- \hp_{e,i}) = (1-p)^{\max\set{2\cq_i(\cpi_i+\sqrt{\csig})\sqrt{n}, \: |\cY_e(i)|}} 
\end{align}
can inductively be made equal and thus independent of the history 
(by only maintaining a weak upper bound on~$\max_e |\cY_e(i)|$; see~\eqref{def:Pi}, \eqref{ob1} and \refL{lemma1}), 
which in turn helps to keep various error terms under control.

\begin{remark}\label{rem:polystep}
Note that each step of our nibble construction requires only randomized polynomial time 
(since we can easily find a \emph{maximal} edge-disjoint collection~$\cD_{i+1} \subseteq \cB_{i+1}$ by a deterministic greedy algorithm). 
\end{remark}

\subsection{Pseudo-random intuition: trajectory equations}\label{sec:pseudo}
In this informal section we give a heuristic explanation of the differential equation 
that predicts the behaviour of~$(\vO_i,\vE_i)$ for $0 \le i \le I \approx n^{\beta}$. 
Inspired by~\cite{K3maximal,Kim}, our main non-rigorous ansatz 
is that the edge-sets~$(\vO_i,\vE_i)$ should resemble properties of 
a random subgraph of~$H$ with two types of edges, where 
\begin{equation}\label{eq:heur}
\P(e \in O_i) \approx q_i \quad \text{ and } \quad \P(e \in E_i) \approx \cpi_i/\sqrt{n} 
\end{equation}
are approximately independent. % for~$e \in E(H)$. 
We now derive properties of~$q_i,\pi_i$ that are consistent with this ansatz. 
For example, combining $E_{i+1} = E_i \cup \Gamma_{i+1}$ with the 
random construction of $\Gamma_{i+1} \subseteq O_i$, we expect to have 
\begin{equation}\label{eq:heur:Ei}
\P(e \in E_{i+1})-\P(e \in E_{i}) 
= \P(e \in \Gamma_{i+1} \mid e \in O_i)\P(e \in O_i) \approx p \cdot \cq_i = \csig \cq_i/\sqrt{n} ,
\end{equation}
which together with~\eqref{eq:heur} and $E_0 = \emptyset$ suggests that 
\begin{equation}\label{eq:heur:pii:value}
\pi_{i+1} - \pi_i \approx  \csig q_i \quad \text{ and } \quad \pi_0 \approx 0 .
\end{equation}
Furthermore, with lots of hand-waving, by~\eqref{def:Oi} we intuitively have $\vO_i \setminus \vO_{i+1} = \Gamma_{i+1}\cup \cC_{i+1} \cup \cC^2_{i+1} \approx \cC_{i+1}$ (since each closed edge in $\cC^2_{i+1}$ requires the presence of at least two random edges from~$\Gamma_{i+1} \subseteq \vO_i$). 
As~\eqref{eq:heur} suggests $\E |Y_e(i)| \lesssim 2 q_i \pi_i \sqrt{n}$, % \lessapprox
by the stabilization mechanism~\eqref{def:pe:cons} and $p=\sigma/\sqrt{n}$ 
we thus loosely expect that 
\[
\P(e \in O_{i+1} \mid O_i,E_i) \approx \P(e \not\in\cC_{i+1} \mid O_i,E_i) = (1-p)^{2\cq_i(\cpi_i+\sqrt{\csig})\sqrt{n}} \approx  1-2\csig\cq_i \cpi_i \quad \text{ for $e \in O_i$},
\]
where we bluntly ignored the~$\sqrt{\csig}$ in the exponent.
Similar to~\eqref{eq:heur:Ei}, using~\eqref{eq:heur} we thus ought to have 
\begin{equation}\label{eq:heur:qi2}
q_{i+1}-q_i \approx \P(e \in O_{i+1})-\P(e \in O_{i}) \approx - 2\csig\cq_i\cpi_i \cdot \P(e \in O_{i}) \approx -2 \sigma q_i^2 \pi_i.
\end{equation} 
To extract the behaviour of~$\pi_I$ from~\eqref{eq:heur:pii:value} and~\eqref{eq:heur:qi2}, 
we further assume that $\pi_i \approx \cPsi(i \csig)$ holds for some smooth function $\cPsi(x)$, where~$\csig \ll 1$ is tiny. 
Using Taylor series, % (and the implicit assumption~$\csig \ll 1$), 
in view of~\eqref{eq:heur:pii:value} and $O_0 = E(H)$ this suggests that 
\begin{equation}\label{eq:heur:qi3}
%\pi_i \approx \cPsi(i \csig), \qquad 
q_i \approx \cPsi'(i \csig) \quad \text{ and } \quad q_0 \approx 1 .
\end{equation}
Together with~\eqref{eq:heur:qi2} and the initial values from~\eqref{eq:heur:pii:value} and~\eqref{eq:heur:qi3}, 
this leads to the second order differential equation $\cPsi''(x) = -2\cPsi'(x)^2\cPsi(x)$ with~$\cPsi'(0)=1$ and~$\cPsi(0)=0$, 
which in turn reduces to the simple ODE 
\begin{equation}\label{eq:heur:qi}
\cPsi'(x)=  e^{-\cPsi^2(x)} \quad \text{and} \quad  \cPsi(0)=0 .
\end{equation}
Noting the implicit solution $x = \int_{0}^{\cPsi(x)}e^{t^2}dt$, 
it now is easy to derive that $\cPsi(x) \approx \sqrt{\log x}$ as $x \to \infty$  
(see, e.g., the proof of~\eqref{ai2} in \refApp{sec:apx}). 
Since $I \approx n^{\beta}$ is sufficiency large compared to $\sigma$ (which will be of form~$\sigma = (\log n)^{-\Theta(1)}$, see~\eqref{def:sigma} in \refS{sec:def}), 
% together with~\eqref{eq:heur:qi3}  \pi_i \approx \cPsi(i \csig)
this makes it plausible that 
\begin{equation}\label{eq:heur:pi_i}
\pi_I \approx \cPsi(I \csig) \approx \sqrt{\log(I \csig)} \approx \sqrt{\beta \log n} .
\end{equation}

Finally, since by construction we expect $|E_{i+1} \setminus E_i| \approx |T_{i+1} \setminus T_i|$ to hold for all $0 \le i < I$, 
the edge-sets~$E_I$ and~$T_I$ ought to share many properties. 
Together with~\eqref{eq:heur} and~\eqref{eq:heur:pi_i} %, for~$e \in E(H)$ 
this intuitively suggests
\begin{equation}\label{eq:heur:Ti}
\P(e \in \vF_I) \approx \P(e \in E_I) \approx \sqrt{\beta (\log n)/n} ,
\end{equation}
making the pseudo-random edge-estimate~\eqref{eq:thmiteration:edge} %of \refT{thmiteration} 
plausible for~$G=(V,\vF_I)$ with $\vF_I \subseteq E_I \subseteq E(H)$.

\subsection{Definitions and parameters}\label{sec:def}
In this section we formally define several variables 
and parameters used in our analysis of the nibble construction. 
We start with two standard notions from graph theory: 
for any edge-subset $S \subseteq \binom{V}{2}$ we write 
\begin{align}
S(A,B) &:= \set{ab \in S: \: a \in A, \: b \in B}, \\
N_S(v) &:= \set{w \in V: \: vw\in S} ,
\end{align}
where $A,B \subseteq V$ 
are vertex-disjoint. 
For all pairs of distinct vertices~$u,v \in V$ 
we then define 
\begin{align}
\label{def:Xuv}
\cX_{uv}(i) &:=  N_{\vO_i}(u)\cap N_{\vO_i}(v), \\
\label{def:Zuv}
\cZ_{uv}(i) &:=  N_{\vE_i}(u)\cap N_{\vE_i}(v) ,
\end{align}
where $|\cX_{uv}(i)|$ and $|\cZ_{uv}(i)|$ intuitively correspond 
to an `open codegree' and the usual codegree, respectively 
(note that~$|\cY_{uv}(i)|$ corresponds to a `mixed codegree', see~\eqref{def:Yuv}).

Guided by \refS{sec:pseudo}, we define $\cPsi(x)$ as the unique solution to the differential equation 
$\cPsi'(x)=\exp(-\cPsi^2(x))$ and $\cPsi(0)=0$ from~\eqref{eq:heur:qi}. 
With the heuristics~\eqref{eq:heur} in mind, we introduce the parameters 
\begin{align}
\label{def:sigma}
\csig &:=(\log n)^{-2}, \\
\label{def:qi}
\cq_i &:= \cPsi'(i\csig) = e^{-\cPsi^2(i\csig)} ,\\ 
\label{def:pii}
\cpi_i &:=\csig+\sum_{j=0}^{i-1}\csig\cq_j=\cpi_{i-1}+\csig\cq_{i-1} \indic{i \ge 1},
\end{align}
making~\eqref{eq:heur:pii:value} and \eqref{eq:heur:qi3} rigorous 
(starting with $\cpi_0=\csig > 0$ leads to cleaner formulae later on). 
With foresight, for $i \le I$ we also introduce the `relative error' parameter  
\begin{align}
\label{def:tau}
\ctau_i &:= 1-\frac{\delta\cpi_i}{2\cpi_I}=\ctau_{i-1}-\frac{\delta\csig\cq_{i-1}}{2\cpi_I}\indic{i \ge 1} ,
\end{align}
which slowly degrades from $\tau_0 = 1- o(\delta)$ to $\tau_I = 1-\delta/2$. 

With an eye on \refT{thmiteration}, for concreteness we introduce the absolute constants\footnote{To make this paper easier to read, we have made no attempt to optimize the constants $D_0,\beta_0$ in~\eqref{def:D0beta0}.} 
\begin{align}\label{def:D0beta0}
D_0 := 108 \quad \text{ and } \quad \beta_0 := 1/14 ,
\end{align}
as well as the set-sizes (with $s_0 \ll s$) and idealized edge-probability 
\begin{align}\label{def:ss0}
s & := \bigceil{C\sqrt{n\log n}}, \qquad s_0 := \bigfloor{\sigma^4\cq_I^2 s},  \quad \text{ and } \quad \rho := \sqrt{\beta(\log n)/n},
\end{align}
and, recalling $O_0=E(H)$, the collection of `relevant' pairs of large vertex-sets 
\begin{align}\label{def:fSs}
\fS_s &:= \{(A,B) : \: \text{disjoint $A,B \subseteq V$ with $|A|=|B|=s$ and $|\vO_0(A,B)|\ge\gamma |A||B|$}\}.
\end{align}

\subsection{Main nibble result: pseudo-random properties}\label{sec:mainnibble}
In this section we state our main nibble result \refT{prop}, 
which implies our main tool \refT{thmiteration} and 
establishes various pseudo-random properties of $(\vO_i,\vE_i,\vF_i,\Gamma_{i})_{0\le i\le I}$. 
The following event is of core interest: 
\begin{align}
\label{def:cTI}
\cT_I & := \bigset{\text{$|\vF_I(A,B)|= (1\pm\delta)\rho |\vO_0(A,B)|$ for all~$(A,B) \in \fS_s$}} .
\end{align}
Indeed, it implies the conclusion of \refT{thmiteration} with~$G=(V,\vF_I)$ since the edge-set $\vF_I \subseteq E_I \subseteq E(H)=\vO_0$ is triangle-free. 
To get a handle on~$\cT_I$, % and thus the triangle-free edge-set~$T_I \subseteq E_I \subseteq O_0$, 
in view of \refS{sec:nibbledetails} it is natural 
that we also require some control over the other edge-sets~$(\vE_i,\vO_i,\Gamma_{i})_{0 \le i \le I}$. 
To this end we introduce the `good' events 
\begin{align}
\cG_i :=\cN_{i} \cap \cP_{i} \cap \cQ^+_i\cap\cQ_i
\quad \text{ and } \quad
\cG_{\le i} :=\bigcap_{0\le j\le i}\cG_j ,
\end{align}
where the following auxiliary events encapsulate various pseudo-random properties: 
\begin{align}
\label{def:Ni}
\cN_{i} & := \bigset{\text{$|N_{\vO_i}(v)|\le \cq_i n$ and $|N_{\Gamma_{i}}(v)| \le 2\sigma \cq_{i-1}\sqrt{n}$ for all~$v \in V$}}, \\
\label{def:Pi}
\cP_{i} & := \bigset{\text{$|\cX_{uv}(i)|\le\cq_i^2n$, $|\cY_{uv}(i)|\le 2\cq_i\pi_i\sqrt{n}$, and $|\cZ_{uv}(i)|\le i(\log n)^9$ for all~$u,v \in V$ with~$u \neq v$}}, \\
\label{def:Qip}
\cQ^+_i & := \bigset{\text{$|\vO_i(A,B)|\le \cq_i|A||B|$ for all disjoint $A,B \subseteq V$ with $|A|,|B| \ge s_0$}}, \\
\label{def:Qi}
\cQ_i & := \bigset{\text{$\ctau_i\cq_i|\vO_0(A,B)|\le|\vO_i(A,B)|\le \cq_i|\vO_0(A,B)|$ for all~$(A,B) \in \fS_s$}}. 
\end{align}
In words, the above events give bounds for 
degree-like variables~($\cN_{i}$),  
codegree-like variables~($\cP_{i}$), 
and the number of open edges~($\cQ^+_i$ and~$\cQ_i$). 
A subtle but important point is that $\cN_{i}$, $\cP_{i}$ and $\cQ^+_i$ 
only guarantee one-sided concentration, i.e., ensure upper bounds but no 
matching lower bounds 
(which can fail badly, for example, $|\cY_{uv}(i)|=0$ holds when~$uv \in E_i$). 
Merely $\cQ_i$ guarantees two-sided concentration, 
which is harder to prove, but  
crucial for establishing the edge-estimate from~$\cT_I$ 
(see the heuristic below \refT{prop}).

With $\ctau_i \approx 1$ and $\vO_0 =E(H) \subseteq E(K_n)$ in mind, most of the bounds in~\eqref{def:cTI} and~\eqref{def:Ni}--\eqref{def:Qi} 
can easily be guessed by the pseudo-random heuristics~\eqref{eq:heur} and~\eqref{eq:heur:Ti} from \refS{sec:pseudo} 
(the~$|N_{\Gamma_{i}}(v)|$-bound is one exception: based on $\E|N_{\Gamma_{i}}(v)|  = p \cdot \E |N_{\vO_{i-1}}(v)|$, 
it contains an extra factor of~$2$ to avoid additive error terms; 
another exception is the $|\cZ_{uv}(i)|$-bound: 
it relaxes the prediction $\E|\cZ_{uv}(i)| \le \pi^2_i = O(\log n)$ for technical reasons).

Inspecting~\eqref{def:Ni}--\eqref{def:Qi} in the special case~$i=0$, it is 
not difficult to see that the good event~$\cG_{0}=\cG_{\le 0}$ always holds  
(by combining~$\cq_0=1 \ge \tau_0$ and~$\csig,\cq_{-1},\cpi_0 \ge 0$ with~$\vE_0 =\vF_0= \Gamma_0=\emptyset$). 

\begin{remark}\label{rem:event0}
The event $\cG_{0}$ holds deterministically for any $n$-vertex host graph~$H$.
\end{remark}

Our main nibble result (which is at the heart of this paper) states that, 
under fairly natural constraints, the pseudo-random events $\cT_I$ and $\cG_{\le I}$ 
both hold with very high probability.   
Recall that $I \approx n^{\beta}$. 
\begin{theorem}[Main nibble result]\label{prop}% 
For all $\gamma,\delta \in (0,1]$, $\beta \in (0,\beta_0)$ and $C \ge D_0 /(\delta^2\sqrt{\beta}\gamma)$, 
the following holds for~$n \ge n_0(\gamma,\delta,\beta,C)$: 
we have $\P(\cT_I\cap\cG_{\le I})\ge 1-n^{-\omega(1)}$ for any $n$-vertex host graph~$H$. 
\end{theorem}	
\begin{proof}[Proof of \refT{thmiteration}]\label{prf:thmiteration}
If the event~$\cT_I$ holds, then the triangle-free graph~$G:=(V,\vF_I)$ has the 
claimed properties by~\eqref{def:cTI}, $V=V(H)$ and $\vF_I \subseteq E_I \subseteq E(H) = \vO_0$, 
so \refT{prop} completes the proof. 
\end{proof}	
\begin{remark}\label{rem:prop:poly}
In view of $I = O(n^{\beta_0})$ and \refR{rem:polystep}, the nibble %from \refS{sec:nibbledetails} 
thus yields a randomized polynomial time algorithm  (with error probability $\le n^{-\omega(1)}$)
for constructing the triangle-free~$G \subseteq H$ from \refT{thmiteration}.  
\end{remark}
\begin{remark}\label{rem:prop:const}
The heuristic edge-estimate~\eqref{eq:heur:Ti} suggests that 
the constraint $C = \Omega(1/(\delta^2\sqrt{\beta}\gamma))$ from \refT{prop} is best possible: 
it would also arise if~$G=(V,\vF_I) \subseteq H$ was a random subgraph with edge-probability~$\rho$.\footnote{For all $(A,B) \in \fS_s$ the expected number of edges between~$A$ and~$B$ would then be at least $\lambda_{A,B} := \E |\vF_I(A,B)| = \rho |\vO_0(A,B)| \ge \rho \cdot \gamma s^2 \ge \sqrt{\beta} \gamma C \cdot s \log n$. 
Using a union bound and standard Chernoff bounds, the probability that~$\cT_I$ from~\eqref{def:cTI} fails would thus be at most $\sum_{(A,B) \in \fS_s} e^{-\Theta(\delta^2 \lambda_{A,B})} \le n^{2s-\Omega(\delta^2\sqrt{\beta} \gamma C s)} = o(1)$ for $C = \Omega(1/(\delta^2\sqrt{\beta}\gamma))$ large enough.\label{fn:prop:const}}
\end{remark}
We defer the proof of~\refT{prop} to \refS{sec:proof}, 
and now just outline a brief heuristic argument that illustrates how the event 
$\cG_{\le I} \subseteq \bigcap_{0\le i\le I}\cQ_i$ is 
instrumental for establishing the edge-estimate from~$\cT_I$ (which seems informative).
Similar to~\eqref{eq:heur:Ti}, in view of \refS{sec:nibbledetails} we expect that 
in each step only few edges are removed due to the creation of triangles, 
which intuitively suggests  
\[
|T_{i+1}(A,B) \setminus T_i| \approx |\vE_{i+1}(A,B) \setminus \vE_i| .
\]
Combining the construction of $\vE_{i+1}\setminus \vE_{i} = \Gamma_{i+1} \subseteq \vO_i$ 
with the event $\cQ_i$ and $\tau_i \approx 1$, 
we also expect that 
\[
|\vE_{i+1}(A,B) \setminus \vE_i| = |\Gamma_{i+1}(A,B)| \approx p \cdot |\vO_i(A,B)| \approx p \cdot q_i |\vO_0(A,B)| .
\]
Recalling~$p=\csig/\sqrt{n}$ and $\rho = \sqrt{\beta(\log n)/n}$, using the definition~\eqref{def:pii} of $\pi_I$ 
and the approximation~$\pi_I \approx \sqrt{\beta \log n}$ from~\eqref{eq:heur:pi_i} 
it now becomes plausible that 
\[
|T_{I}(A,B)| = \sum_{0 \le i < I} |T_{i+1}(A,B) \setminus T_i| 
\approx \frac{\sum_{0 \le i < I}\csig q_i}{\sqrt{n}} \cdot |\vO_0(A,B)| 
\approx \frac{\pi_I}{\sqrt{n}}\cdot |\vO_0(A,B)| 
\approx \rho |\vO_0(A,B)|, 
\]
as suggested by~$\cT_I$ (\refS{sec:P5} contains a rigorous version of this heuristic argument).

\subsection{Tools and auxiliary estimates}
In this preparatory section we gather, for later reference, some 
results that will be used throughout the proof of \refT{prop} 
(mostly probabilistic and combinatorial tools, and ending with some auxiliary estimates). 
On a first reading the reader may perhaps wish to skip straight to Section~\ref{sec:proof}.

We start with a convenient version of the bounded differences inequality~\cite{McDiarmid1989,McDiarmid1998,TBDI} for Bernoulli variables. 
Note that the upper tail estimate~\eqref{eq:BDImon} for decreasing functions does \emph{not} have an extra $Ct$ term in the exponent like~\eqref{eq:BDI}. 
Remarks~\ref{rem:BDI:LT}--\ref{rem:chernoff} are well-known, see, e.g., \cite[Theorem~2.3, 3.8, and~3.9]{McDiarmid1998} or \cite[Corollary~1.4]{TBDI}.   
Inequality~\eqref{eq:BDImon} can be deduced from the arguments in~\cite[Lemma~7.14]{McDiarmid1989}, but this monotone version does not seem to be widely known; 
in \refApp{sec:apx} we thus include a simple proof for completeness.   
\begin{theorem}\label{thm:BDI}% [Bounded differences inequality for Bernoulli variables]
Let $(\xi_\alpha)_{\alpha \in \cI}$ be a finite family of independent random variables with~$\xi_\alpha \in \{0,1\}$.  
Let $f:\{0,1\}^{|\cI|} \to \RR$ be a function, and 
assume that there exist numbers $(c_\alpha)_{\alpha \in \cI}$ such that the 
following holds for all $z=(z_{\alpha})_{\alpha \in \cI} \in \{0,1\}^{|\cI|}$ and $z'=(z'_{\alpha})_{\alpha \in \cI} \in \{0,1\}^{|\cI|}$: 
$|f(z)-f(z')| \le c_{\beta}$ if $z_{\alpha} = z'_{\alpha}$ for all~$\alpha \neq \beta$. 
Define $X:= f\bigl((\xi_\alpha)_{\alpha \in \cI}\bigr)$ and $\lambda := \sum_{\alpha \in \cI}c_\alpha^2 \P(\xi_{\alpha}=1)$.
Then, for all $t \ge 0$, 
\begin{equation}\label{eq:BDImon}
\P(X\ge \E X  + t) \le\exp\biggpar{-\frac{t^2}{2\lambda}} 
\end{equation}
if the function $f$ is decreasing (i.e., that $f(z) \le f(z')$ whenever $z_{\alpha} \ge z'_{\alpha}$ for all $\alpha \in \cI$). 
\end{theorem}
\begin{remark}\label{rem:BDI:LT} % [Lower tail]
Define $C:=\max_{\alpha \in \cI}c_\alpha$. 
If we drop the assumption that~$f$ is decreasing, then 
\begin{equation}\label{eq:BDI}
\P(X\le \E X - t) \le\exp\biggpar{-\frac{t^2}{2(\lambda+C t)}} .
\end{equation}
\end{remark}
\begin{remark}\label{rem:chernoff}% [Chernoff bounds]
In the special case $X = \sum_{\alpha \in \cI}\xi_\alpha$ we have $C=c_\alpha=1$ and $\lambda=\E X$. 
Standard Chernoff bounds (or applying~\eqref{eq:BDImon}-\eqref{eq:BDI} to the decreasing function~$-X$) 
then 
show that 
in this case 
$\P(X\le \E X - t)$ and $\P(X\ge \E X + t)$ are at most the right hand side of~\eqref{eq:BDImon} and~\eqref{eq:BDI}, respectively. 
\end{remark}

For random variables with a special combinatorial form (based on the occurrence of events with `limited overlaps') 
we shall use the following Chernoff-type upper tail inequality, which is a convenient corollary of a more general result by Warnke~\cite[Theorem 9]{APUT}. 
Note that the exponent of~\eqref{eq:C} scales with~$1/C$. 
\begin{theorem}\label{thm:UT}
Let $(\xi_i)_{i \in \fS}$ be a finite family of independent random variables with~$\xi_i \in \{0,1\}$. 
Let $(Y_{\alpha})_{\alpha \in \cI}$ be a finite family of variables $Y_{\alpha} := \indic{\xi_i=1 \text{ for all } i \in \alpha}$ with $\sum_{\alpha \in \cI} \E Y_{\alpha} \le \mu$. 
Define $Z_C :=\max \sum_{\alpha \in \cJ} Y_{\alpha}$, where the maximum is taken over 
all $\cJ \subseteq \cI$ with $\max_{\beta \in \cJ}|\set{\alpha \in \cJ: \alpha \cap \beta \neq \emptyset}| \leq C$. 
%Set $\varphi(x)=(1+x)\log(1+x)-x$.  
Then, for all $C,t>0$, 
\begin{equation}\label{eq:C}
\P(Z_C \geq \mu +t) 
%\le \exp\left(-\frac{\varphi(t/\mu)\mu}{C} \right) =
\le e^{-\mu/C} \cdot \biggpar{\frac{e\mu}{\mu+t}}^{(\mu+t)/C} 
\le \exp\biggpar{-\frac{t^2}{2C(\mu +t)}} .
\end{equation}
\end{theorem}

The following simple combinatorial lemma formalizes the intuition 
that we expect $\sum_i|U_i| = O(|U|)$ whenever the subsets 
$U_i\subseteq U$ are nearly disjoint (i.e., have small pairwise intersections).
\begin{lemma}\label{lem:overlap}
Suppose that $(U_i)_{i\in \cI}$ is a family of subsets $U_i\subseteq U$ with $|U_i|\ge z > 0$ 
and $|U_i\cap U_j|\le y$ for all~$i\neq j$. 
Then $z\ge \sqrt{4|U|y}$ implies 
$|\cI|\le 2|U|/z$ and $\sum_{i\in \cI}|U_i|\le 2|U|$.  
\end{lemma}

\begin{proof} 
Aiming at a contradiction, suppose that $|\cI| > 2|U|/z$. 
Then there is $\cJ \subseteq \cI$ with 
$|\cJ| = \lfloor2|U|/z\rfloor+1$. 
Observe that, for any $i \in \cJ$, 
\begin{equation}\label{eq:SiSjlb}
\sum_{j\in \cJ: i \neq j}|U_j\cap U_i|\le (|\cJ|-1) y \le 2|U|y/z \le z/2 \le |U_i|/2 .
\end{equation}
So we obtain a contradiction by noting that 
\begin{equation}\label{eq:Slb}
|U|\ge \bigl|\bigcup_{i \in \cJ} U_i\bigr|
\ge \sum_{i \in \cJ} \Bigpar{|U_i| - \sum_{j\in \cJ: i \neq j}|U_j\cap U_i|} 
\ge \sum_{i \in \cJ} |U_i|/2 \ge |\cJ| z/2 > |U| .
\end{equation}

With $|\cI| \le 2|U|/z$ in hand, 
after replacing $\cJ$ with $\cI$, note that~\eqref{eq:SiSjlb}  
and the first three inequalities of~\eqref{eq:Slb} remain valid, 
completing the proof of $\sum_{i\in \cI}|U_i|\le 2|U|$. 
\end{proof}

Our final auxiliary result contains a number of convenient estimates 
involving the parameters $\cq_i, \cpi_i, \csig,I$ defined in \refS{sec:def}. 
Roughly put, \eqref{ai5}--\eqref{ai2} state that $\cq_i \approx \cq_{i+1}$, $1-2\csig\cq_i\cpi_i \approx q_{i+1}/q_i$ and $\cpi_I \approx \sqrt{\log(I\csig)}$,  
as predicted by~\eqref{eq:heur:qi2} and~\eqref{eq:heur:pi_i}. 
The technical estimates~\eqref{ai1}--\eqref{ai7} can safely be ignored on a first reading. 
The proof of \refL{lem:aux} is based on elementary calculus and thus deferred to \refApp{sec:apx} 
(it also establishes~$q_i \ge q_I = n^{-\beta + o(1)}$, which together with~$I \approx n^{\beta}$ and~\eqref{ai7} 
motivates our choice of~$\beta_0=1/14$). 
\begin{lemma}\label{lem:aux}
If $\beta \in (0, \beta_0)$, 
then there exists $\tau,n_0 > 0$ such that, 
for all $n \ge n_0$ and $0 \le i \le I$, 
\begin{gather}
\label{ai1}
\max\Bigset{\max_{j \in \set{0,1,2}}\bigset{q_i \cpi_i^j}, \: \sqrt{\csig}\cpi_i} \le 1,   \\
\label{ai7}
\min\Bigset{ \min_{j \in [4]}\bigset{q_i^j \sqrt{n}}, \: q_i^2 \sqrt{n}/I, \: q_i^3 \sqrt[4]{n}/\sqrt{I}} \ge n^{\tau} , \\
\label{ai5}
0 \le \cq_i-\cq_{i+1}\le 4\csig\cdot\min\{\cq_{i}, \: \cq_{i+1}, \: \cq_i \cpi_i\}, \\
\label{ai4}
\bigabs{(1-2\csig\cq_i\cpi_i)-q_{i+1}/q_i}\le 8\csig^2\cq_i,\\
\label{ai2}
\bigabs{\cpi_I-\sqrt{\log(I\csig)}} \le 2 .
\end{gather}
\end{lemma}
As a simple application, for $0 \le i \le I$ 
we now bound the stabilization probability~$\hp_{e,i}$ defined in~\eqref{def:pe}. 
Since~\eqref{ai7} implies $\cq_i\sqrt{\csig}\sqrt{n} \gg 1$, 
by applying $(1-p)^r \ge 1-pr = 1-\csig r/\sqrt{n}$ (valid for $r \ge 1$) we infer 
\begin{equation}\label{eq:hpei}
\hp_{e,i} \le 1-(1-p)^{2\cq_i(\cpi_i+\sqrt{\csig})\sqrt{n}} \le  2\csig\cq_i(\cpi_i+\sqrt{\csig}) \le \cq_i ,
\end{equation}
where we used $\sqrt{\csig}\cpi_i \le 1$ and $\csig \ll 1$ (see~\eqref{ai1} and~\eqref{def:sigma}) for the last inequality.

\section{Analyzing the nibble}\label{sec:proof} 
In this section we prove our main nibble result \refT{prop} 
(which in turn implies our main tool \refT{thmiteration}, see \refS{sec:mainnibble})
as a corollary of the following auxiliary lemma. 
\begin{lemma}\label{lem:prop}
Under the assumptions of \refT{prop}, for~$n \ge n_0(\gamma,\delta,\beta,C)$ we have 
\begin{align}
\label{eq:goal1}
\P(\neg \cG_{i+1} \mid \cG_{\le i}) & \le n^{-\omega(1)} \qquad \text{for all $0 \le i \le I-1$,}\\
\label{eq:goal2}
\P(\neg \cT_I \cap \cG_{\le I}) & \le n^{-\omega(1)} .
\end{align} 
\end{lemma}
\begin{proof}[Proof of \refT{prop}]
Recalling $I \le \lceil n^{\beta_0}\rceil = n^{O(1)}$ and $\cG_{\le i}=\bigcap_{0\le j\le i}\cG_j$, 
note that $\P(\neg \cG_{0})=0$ (see \refR{rem:event0}) and~\eqref{eq:goal1} readily imply 
$\P(\neg\cG_{\le I})\le n^{-\omega(1)}$, 
which together with~\eqref{eq:goal2} completes the proof. 
\end{proof}	
The remainder of this section is devoted to the proof of \refL{lem:prop}: 
the proof of~\eqref{eq:goal1} with~$\neg\cG_{i+1} = \neg\cN_{i+1} \cup \neg\cP_{i+1} \cup \neg\cQ^+_{i+1}\cup \neg\cQ_{i+1}$  is spread across Sections~\ref{sec:P1}--\ref{sec:P3P4},  
and the proof of~\eqref{eq:goal2} is given in \refS{sec:P5}.

\subsection{Preliminaries: setup and conventions}\label{sec:setupconvetions} 
To avoid clutter, up to (and including) \refS{sec:P3P4} we shall suppress the conditioning in the notation: 
we will write $\P(\cdot)$ and $\E(\cdot)$ as shorthand for $\P(\cdot \mid \cF_i)$ and $\E(\cdot \mid \cF_i)$, 
where $(\cF_i)_{0 \le i \le I}$ denotes the natural filtration associated with $(\vO_i,\vE_i,\vF_i,\Gamma_i,\cW_i)_{0\le i\le I}$, as usual. 
We will also tacitly assume that the $\cF_i$-measurable event~$\cG_{\le i}$ holds. 
Conditional on $\cF_i$, note that by construction of the random edge-sets $\Gamma_{i+1}$ and $\cW_{i+1}$, 
the (conditional) probability space formally consists of the $2 |\vO_i|$ independent Bernoulli random variables 
$\xpar{\indic{e \in \Gamma_{i+1}}, \indic{e \in \cW_{i+1}}}_{e \in \vO_i}$, 
with $\P(e \in \Gamma_{i+1})= p = \sigma/\sqrt{n}$ and $\P(e \in \cW_{i+1}) = \hp_{e,i} \le q_i$, see~\eqref{eq:hpei}.

Using the above setup and conventions, we shall repeatedly consider random variables of form 
\begin{equation}\label{eq:Xf}
X=f\Bigpar{\bigpar{\indic{e \in \Gamma_{i+1}}, \indic{e \in \cW_{i+1}}}_{e \in \vO_i}}. 
\end{equation}
To get a handle on the (conditional) expectation~$\E X$ we will often use $\vO_{i+1} \subseteq \vO_i \setminus \cC_{i+1}$ together with the following key lemma, 
which hinges on the stabilization mechanism 
to equalize all (conditional) probabilities~$\P(e\not\in\cC_{i+1})$, 
see~\eqref{ob1} below. 
(The extra $\sqrt{\csig}$ term in~\eqref{def:pe} ensures 
that~$\P(e\not\in\cC_{i+1}) < \cq_{i+1}/\cq_i$ 
holds with plenty of elbow room, which is convenient for avoiding ugly error terms in the upper bounds of~\eqref{def:Ni}--\eqref{def:Qi}.)  
\begin{lemma}\label{lemma1}
We have $\P(e\not\in\cC_{i+1}) - \cq_{i+1}/\cq_i \in [-3\csig^{3/2}\cq_i, \:  -\csig^{3/2}\cq_i]$ for all $e\in \vO_i$. 
\end{lemma}
\begin{proof}
For any $e \in \vO_i$, since $|\cY_e(i)| \le 2\cq_i\cpi_i\sqrt{n}$ by $\cG_{\le i} \subseteq \cP_i$, 
by definition of $\cC_{i+1} = \cC^1_{i+1} \cup S_{i+1}$ we have 
\begin{equation}\label{ob1}  
\P(e\not\in\cC_{i+1})=\P(e\not\in \cC^1_{i+1}) \cdot \P(e\not\in S_{i+1}) = (1-p)^{|\cY_e(i)|} \cdot (1-\hp_{e,i})=(1-p)^{2\cq_i(\cpi_i+\sqrt{\csig})\sqrt{n}} .           
\end{equation}
It is well-known (and easy to check) that $1-rx\le(1-x)^r\le1-rx+\binom{r}{2}x^2$ for all~$x \in [0,1]$ and~$r\ge 2$.  
Using $\sqrt{n}p=\sigma \ll 1$ and $\max\{\cq_i,\cq_i\cpi_i,\cq_i\cpi_i^2\} \le 1$ (see~\eqref{ai1}), it follows that 
\begin{equation*}%\label{ob2}  
\Bigabs{\P(e\not\in\cC_{i+1})-\bigsqpar{1-2\csig\cq_i(\cpi_i+\sqrt{\csig})}} \le 2\csig^2 \cq_i^2(\cpi_i+\sqrt{\csig})^2 = O( \csig^2\cq_i) = o(\csig^{3/2}\cq_i) .
\end{equation*}
This completes the proof %(for~$n$ large enough) 
since $1-2\csig\cq_i\cpi_i = q_{i+1}/q_i + o(\csig^{3/2}\cq_i)$ by~\eqref{ai4}. 
\end{proof}
To deduce concentration properties of such random variables~$X$ we shall frequently rely on the bounded differences inequality (\refT{thm:BDI}). 
In those cases we will bound the associated parameter~$\lambda$ via  
\begin{equation}\label{eq:l}
\lambda = \sum_{e \in \vO_i} c_e^2 \P(e \in \Gamma_{i+1}) + \sum_{e \in \vO_i} \hc_e^2 \P(e \in \cW_{i+1}) \le p \sum_{e \in \vO_i} c_e^2  + \cq_i\sum_{e \in \vO_i} \hc_e^2 ,
\end{equation}
where the \emph{edge-effect} $c_e$ is an upper bound on how much~$X$ can change if we modify the indicator $\indic{e \in \Gamma_{i+1}}$ (alter whether~$e$ is in $\Gamma_{i+1}$ or not), 
and the  \emph{stabilization-effect} $\hc_e$ is an upper bound on how much~$X$ can change if we modify the indicator $\indic{e \in \cW_{i+1}}$ (alter whether~$e$ is in $\cW_{i+1}$ or not). 
Moreover, the following simple observation will later allow us to control the above sum~\eqref{eq:l} of these effects. 
\begin{lemma}\label{summinglemma}
If $\cG_{\le i}$ holds, then $\sum_{e \in \vO_i}|\cY_e(i) \cap J| \le 2\cq_i\cpi_i \sqrt{n} \cdot |J| $ for any edge-subset $J \subseteq \binom{V}{2}$.  
\end{lemma}
\begin{proof}
For any $e \in \vO_i$, note that $f \in \cY_e(i)$ implies $e \in \cY_f(i)$. 
As $\cY_f(i) \subseteq \vO_i$, we infer 
\begin{equation*}
\begin{split}
\sum_{e \in \vO_i}|\cY_e(i) \cap J| & = \sum_{f \in J}\sum_{e \in \vO_i}\indic{f \in \cY_e(i)} \le \sum_{f \in J}\sum_{e \in \vO_i}\indic{e \in \cY_f(i)} =  \sum_{f \in J} |\cY_f(i)| .
\end{split}
\end{equation*}
This completes the proof since $\cG_{\le i} \subseteq \cP_i$ implies $|\cY_f(i)| \le 2\cq_i\cpi_i\sqrt{n}$. 
\end{proof}

\subsection{Event $\cN_{i+1}$: degree-like variables $|N_{\vO_{i+1}}(v)|$ and $|N_{\Gamma_{i+1}}(v)|$}\label{sec:P1}  
\begin{lemma}\label{lem:Pro1}
We have 
$\P(\neg \cN_{i+1})\le n^{-\omega(1)}$. 
\end{lemma}
\begin{proof}
We start with $|N_{\vO_{i+1}}(v)|$.  
Note that $\vO_{i+1} \subseteq \vO_i \setminus \cC_{i+1}$ implies 
\begin{equation}\label{def:X1O}
|N_{\vO_{i+1}}(v)|\le\sum_{w\in N_{\vO_i}(v)}\indic{vw\not\in\cC_{i+1}} =:X.
\end{equation}
Since $\cG_{\le i} \subseteq \cN_{i}$ implies $|N_{\vO_i}(v)|\le\cq_i n$, using \refL{lemma1} we obtain 
\begin{equation}\label{eq:EXO}
\E X 
 = \sum_{w\in N_{\vO_i}(v)}\P(vw\not\in\cC_{i+1}) 
\le |N_{\vO_i}(v)| \cdot(\cq_{i+1}/\cq_i-\csig^{3/2}\cq_i)
\le \cq_{i+1}n-\csig^{3/2}\cq_i^2n.
\end{equation}
Gearing up to apply \refT{thm:BDI} to~$X$, 
we now bound the associated parameter $\lambda \le p \sum_{e \in \vO_i} c_e^2  + q_i \sum_{e \in \vO_i} \hc_e^2$ from~\eqref{eq:l}.
Set $\fX_v:=\{v\} \times  N_{\vO_i}(v) \subseteq \vO_i$, and recall that~$\cC_{i+1} = \cC^1_{i+1} \cup \cW_{i+1}$, 
where~$\cC^1_{i+1}$ depends only on~$\Gamma_{i+1}$ and thus is independent of~$\cW_{i+1}$. 
The edge-effect~$c_e$ 
(an upper bound on how much~$X$ changes if we alter whether~$e \in \Gamma_{i+1}$ or~$e \not\in \Gamma_{i+1}$) 
is thus at most the number of changes to 
$\cC^1_{i+1} \cap \fX_v=\set{vw \in \fX_v : \cY_{vw}(i) \cap \Gamma_{i+1} \neq \emptyset}$. 
Since $e \in \cY_{vw}(i)$ implies $vw \in \cY_e(i)$ when $vw \in \fX_v$, 
we infer $c_e \le |\cY_e(i)\cap \fX_v |\le |\cY_e(i)|\le 2\cq_i\cpi_i\sqrt{n}$ by~$\cG_{\le i} \subseteq \cP_i$. 
Using \refL{summinglemma}, $|\fX_v|  = |N_{\vO_i}(v)| \le q_i n$, and $\cq_i\cpi_i^2 \le 1$ (see~\eqref{ai1}), it follows that
\begin{equation}\label{eq:SL1}
p\sum_{e \in \vO_i}c_e^2 
\le p \cdot 2\cq_i\cpi_i\sqrt{n}\cdot \sum_{e \in \vO_i}|\cY_e(i)\cap \fX_v| 
\le \csig/\sqrt{n} \cdot (2\cq_i\cpi_i\sqrt{n})^2 \cdot |\fX_v|
\le 4 \csig \cq_i^3\cpi_i^2 n^{3/2}
\le 4 \csig  \cq_i^2 n^{3/2} .
\end{equation}
By our above discussion, the stabilization-effect~$\hat{c}_e$ 
(an upper bound on how much~$X$ changes if we alter whether~$e \in \cW_{i+1}$ or~$e \not\in \cW_{i+1}$) 
is at most the number of changes to $\cW_{i+1} \cap \fX_v$. 
Hence $\hat{c}_e \le\indic{e\in\fX_v}$, so that 
\[
q_i\sum_{e \in \vO_i} \hat{c}_e^2\le q_i \cdot |\fX_v| \le \cq_i^2 n \ll \csig \cq_i^2 n^{3/2} .
\] 
Noting that~$X$ is a %defined in~\eqref{def:X1O} is a 
decreasing function of the edge-indicators~$(\indic{e \in \Gamma_{i+1}}, \indic{e \in \cW_{i+1}})_{e \in \vO_i}$, 
using \refT{thm:BDI} together with the $\lambda$--bound~\eqref{eq:l} 
and $q_i^2 n^{1/2} \ge n^{\tau}$ (see~\eqref{ai7}) it follows that  
\[
\P(|N_{\vO_{i+1}}(v)| \ge \cq_{i+1}n)
\le \P(X\ge \E X + \csig^{3/2}\cq_i^2n)
\le \exp\lrpar{-\frac{\csig^{3}\cq_i^4n^2}{2 \cdot 5\csig \cq_i^2  n^{3/2}}}
\le n^{-\omega(1)} .
\]
Taking a union bound over all vertices~$v$ completes the proof for the $|N_{\vO_{i+1}}(v)|$ variables.

Finally, note that $|N_{\Gamma_{i+1}}(v)|$ is a sum of independent Bernoulli random variables with 
\[
\E |N_{\Gamma_{i+1}}(v)| = |N_{\vO_i}(v)| \cdot p \le \cq_i n \cdot \csig/\sqrt{n}=\csig\cq_i\sqrt{n} = : \mu , 
\]
where we used $\cG_{\le i} \subseteq \cN_i$ to bound $|N_{\vO_i}(v)| \le \cq_i n$. 
Applying standard Chernoff bounds (see, e.g., Remark~\ref{rem:chernoff}), 
using $q_i \sqrt{n} \ge n^{\tau}$ (see~\eqref{ai7}) 
it is routine to deduce that $\mu \gg \log n$ and 
\begin{equation*}\label{eq:NGammai}
\P(|N_{\Gamma_{i+1}}(v)| \ge 2\csig\cq_i\sqrt{n}) 
= \P(|N_{\Gamma_{i+1}}(v)| \ge 2\mu) 
\le \exp\lrpar{-\frac{\mu^2}{2 \cdot 2\mu}}
= \exp\lrpar{-\frac{\mu}{4}}
\le n^{-\omega(1)} .
\end{equation*}
Taking a union bound over all vertices~$v$ completes the proof for the $|N_{\Gamma_{i+1}}(v)|$ variables.  
\end{proof}

\subsection{Event $\cP_{i+1}$: codegree-like variables $|\cX_{uv}(i+1)|$, $|\cY_{uv}(i+1)|$ and $|\cZ_{uv}(i+1)|$}\label{sec:codegree}
\begin{lemma}\label{lem:Pro2}
We have 
$\P(\neg \cP_{i+1})\le n^{-\omega(1)}$. 
\end{lemma}
\begin{proof}
We start with~$|\cX_{uv}(i+1)|$. Recalling $\vO_{i+1} \subseteq \vO_i \setminus \cC_{i+1}$, note that by construction 
\begin{equation}\label{def:X}
|\cX_{uv}(i+1)|\le\sum_{w\in\cX_{uv}(i)}\indic{uw\not\in\cC_{i+1} \text{ and } vw\not\in\cC_{i+1}}=:X.
\end{equation}
To estimate $\E X$, note that the event $f \not\in \cC^1_{i+1} = \set{f \in \vO_i : \cY_f(i) \cap \Gamma_{i+1} \neq \emptyset}$ 
is determined by the values of the independent indicator variables $(\indic{e \in \Gamma_{i+1}})_{e \in \cY_f(i)}$. 
In view of the reasoning~\eqref{ob1} for the value of~$\P(e\not\in\cC_{i+1})$, it follows by construction of~$\cC_{i+1} = \cC^1_{i+1} \cup \cW_{i+1}$ that
\begin{equation}\label{eq:PCond} 
\P(uw\not\in\cC_{i+1} \text{ and } vw\not\in\cC_{i+1}) 
 = \P(uw\not\in\cC_{i+1}) \P(vw\not\in\cC_{i+1}) \cdot (1-p)^{-|\cY_{uw}(i) \cap \cY_{vw}(i)|} .
\end{equation}
Since $\cG_{\le i}\subseteq \cP_i$ implies $|\cY_{uw}(i) \cap \cY_{vw}(i)| \le |\cZ_{uv}(i)| \le I (\log n)^9 $ and $|\cX_{uv}(i)| \le \cq_i^2n$,
by combining~\eqref{eq:PCond} with \refL{lemma1} it follows that 
\begin{equation}\label{eq:Xuv}
\E X \le |\cX_{uv}(i)| \cdot (\cq_{i+1}/\cq_{i}-\csig^{3/2}\cq_{i})^2 \cdot (1-p)^{-I(\log n)^9}  \le \cq_{i+1}^2n-\csig^{3/2}\cq_i^3n, 
\end{equation}
where for the last inequality we used $pI(\log n)^9 \ll \csig^{3/2}\cq_i^3 \ll 1$ (since $q_i^3 \sqrt{n}/I \ge n^{\tau}$ by~\eqref{ai7}) and 
$\csig^3\cq_i^4\ll \csig^{3/2}\cq_i^3 \sim \csig^{3/2}\cq_{i+1}\cq_i^2$ (see~\eqref{ai1}--\eqref{ai5}). 
With an eye on \refT{thm:BDI}, we now bound the parameter $\lambda \le p \sum_{e \in \vO_i} c_e^2  + q_i\sum_{e \in \vO_i} \hc_e^2$ from~\eqref{eq:l}.
Set $\fX_{uv}:=\{u,v\}\times \cX_{uv}(i) \subseteq \vO_i$. 
Analogous to the proof of \refL{lem:Pro1} for $|N_{\vO_{i+1}}(v)|$, 
here we have edge-effect $c_e \le |\cY_e(i)\cap \fX_{uv}|\le |\cY_e(i)| \le 2\cq_i\cpi_i\sqrt{n}$ and stabilization-effect $\hat{c}_e \le\indic{e\in\fX_{uv}}$.
Similar to~\eqref{eq:SL1}, using \refL{summinglemma}, $|\fX_{uv}| = 2 \cdot |\cX_{uv}(i)| \le 2 \cq_i^2n$ and $\cq_i\cpi_i^2 \le 1$ it follows that 
\begin{equation}\label{eq:SL2}
p \sum_{e \in \vO_i} c_e^2 \le \csig/\sqrt{n} \cdot (2\cq_i\cpi_i\sqrt{n})^2 \cdot |\fX_{uv}|\le 8\csig\cq_i^4\cpi_i^2n^{3/2}\le 8\csig\cq_i^3 n^{3/2} .
\end{equation}
Furthermore, $q_i\sum \hat{c}_e^2\le q_i|\fX_{uv}|\le 2\cq_i^3n \ll \sigma \cq_i^3 n^{3/2}$. 
Noting that~$X$ %defined in~\eqref{def:X} 
is a decreasing function of the edge-indicators~$(\indic{e \in \Gamma_{i+1}}, \indic{e \in \cW_{i+1}})_{e \in \vO_i}$, 
using \refT{thm:BDI} and $q_i^3 n^{1/2} \ge n^{\tau}$ (see~\eqref{ai7}) it follows that 
\[\P(|X_{uv}(i+1)| \ge \cq_{i+1}^2n)\le\P(X \ge \E X + \csig^{3/2}\cq_i^3n)\le\exp\lrpar{-\frac{\csig^3\cq_i^6n^2}{2 \cdot 9\csig\cq_i^3 n^{3/2}}} \le n^{-\omega(1)} . \] 
Taking a union bound over all pairs of vertices~$u,v$ completes the proof for the $|\cX_{uv}(i+1)|$ variables.

Turning to the more involved $|\cY_{uv}(i+1)|$ variables, note that by construction 
\begin{equation}\label{def:Y}
|\cY_{uv}(i+1)| \le \sum_{w\in \cX_{uv}(i)}\indic{uw\in\Gamma_{i+1}\text{ or } vw\in\Gamma_{i+1}} + \sum_{f \in \cY_{uv}(i)}\indic{f\not\in\cC_{i+1}} =: Y_{uv}^+ + Y_{uv}^*.
\end{equation}
(To clarify: $Y_{uv}^+$ and $Y_{uv}^*$ are defined by the first and second sum in~\eqref{def:Y}, respectively.) 
Using \refL{lemma1} 
together with $\csig \cq_i^2 = \csig \cq_i\cq_{i+1} + o(\csig^{3/2}\cq_i^2 \cpi_i)$ (see~\eqref{ai5}) 
and $\cpi_i\cq_{i+1} = \cq_{i+1}\cpi_{i+1}-\csig\cq_i\cq_{i+1}$ (as~$\pi_{i+1} = \pi_i + \csig \cq_i$ by~\eqref{def:pii}), it follows that 
\begin{equation}
\label{eq:EY}
\begin{split}
\E(Y_{uv}^+ + Y_{uv}^*)
&\le |\cX_{uv}(i)| \cdot 2p \: + \: |\cY_{uv}(i)| \cdot (\cq_{i+1}/\cq_i-\csig^{3/2}\cq_i)\\
& \le 2\csig\cq_i^2 \sqrt{n} + 2\cpi_i\sqrt{n}(\cq_{i+1}-\csig^{3/2}\cq_i^2) 
\le 2\cq_{i+1}\cpi_{i+1}\sqrt{n} -\csig^{3/2}\cq_i^2\cpi_i\sqrt{n} .	
\end{split}
\end{equation}
We now estimate $Y_{uv}^+$ and $Y_{uv}^*$ separately. 
Noting $\E Y^+_{uv} \le 2\csig\cq_i^2\sqrt{n}$ and $\csig^{2}\cq_i^2\cpi_i\sqrt{n} = o(\csig\cq_i^2\sqrt{n})$ (see~\eqref{ai1}), 
using standard Chernoff bounds together with $\cpi_i^2 \ge \cpi_0^2=\csig^2$ %(by the definition~\eqref{def:pii} of~$\pi_i$) 
and $q_i^2 \sqrt{n} \ge n^{\tau}$ (see~\eqref{ai7}) it follows that 
\begin{equation}\label{Y^+}
\P(Y_{uv}^+ \ge \E Y_{uv}^+ + \csig^{2}\cq_i^2\cpi_i\sqrt{n})
\le \exp\lrpar{-\frac{\bigpar{\csig^{2}\cq_i^2\cpi_i\sqrt{n}}^2}{4 \cdot 2\csig\cq_i^2\sqrt{n}}}
\le \exp\lrpar{-\frac{\csig^5\cq_i^2 \sqrt{n}}{8}}
\le n^{-\omega(1)} . 
\end{equation}
For $Y_{uv}^*$ we shall apply \refT{thm:BDI}, 
and we thus now bound $\lambda \le p \sum_{e \in \vO_i} c_e^2  + q_i\sum_{e \in \vO_i} \hc_e^2$ from~\eqref{eq:l}. 
As usual, we have edge-effect $c_e \le |\cY_e(i)\cap \cY_{uv}(i)|\le |\cY_e(i)| \le 2\cq_i\cpi_i\sqrt{n}$ and stabilization-effect $\hat{c}_e \le\indic{e\in \cY_{uv}(i)}$.
Here we can significantly improve the simple worst case estimate $c_e\le |\cY_e(i)|$ when $e \neq uv$. 
Indeed, if $e=w_1w_2$ does not intersect $uv$, then $c_e \le 4$ since $\cY_e(i)\cap \cY_{uv}(i) \subseteq \{u,v\} \times \{w_1,w_2\}$, say. 
Furthermore, if $e=w_1w_2$ intersects $uv$ in one vertex, say $u=w_1$, then $c_e \le \max_f |\cZ_{f}(i)| \le I (\log n)^9$ since $\cY_e(i)\cap \cY_{uv}(i) \subseteq \{u\} \times [ N_{E_i}(w_2) \cap  N_{E_i}(v)]$. 
To sum up, for $e \neq uv$ we have $c_e \le \max\{4,I (\log n)^9\} \le \sigma^{-5} I$, say. 
Similar to~\eqref{eq:SL1} and~\eqref{eq:SL2}, using \refL{summinglemma} and $|\cY_{uv}(i)| \le 2\cq_i\cpi_i\sqrt{n}$  it follows that 
\begin{equation*}%\label{eq:SL3}
p \sum_{e \in \vO_i} c_e^2\le \sigma/\sqrt{n} \cdot \Bigpar{(2\cq_i\cpi_i\sqrt{n})^2 +  \sigma^{-5} I 
\cdot 2\cq_i\cpi_i\sqrt{n} \cdot |\cY_{uv}(i)|} \le 8 \csig^{-4} \cq_i^2 \cpi_i^2 I\sqrt{n}.
\end{equation*}
Furthermore, using $\cpi_i \ge \csig$ and $I \ge 1$ we obtain $q_i\sum_{e \in \vO_i} \hat{c}_e^2 \le q_i|\cY_{uv}(i)| \le 2\cq_i^2\cpi_i\sqrt{n} \ll \csig^{-4} q_i^2 \cpi_i^2 I \sqrt{n}$.
Noting that~$Y^*_{uv}$ %defined in~\eqref{def:Y} 
is decreasing, using \refT{thm:BDI} 
and $\cq_i^2\sqrt{n}/I \ge n^{\tau}$ (see~\eqref{ai7})
it follows that 
\begin{equation}\label{Y^*}
\P(Y_{uv}^* \ge \E Y_{uv}^* + \csig^{2}\cq_i^2\cpi_i\sqrt{n})
\le \exp\lrpar{-\frac{\csig^4\cq_i^4\cpi_i^2 n}{2 \cdot 9 \csig^{-4} \cq_i^2 \cpi_i^2 I\sqrt{n}}} \le n^{-\omega(1)}. 
\end{equation}
Combining the probability estimates~\eqref{Y^+} and~\eqref{Y^*} with inequalities~\eqref{def:Y}--\eqref{eq:EY} and $\csig^2 \ll \csig^{3/2}$, 
now a union bound argument (to account for all pairs of vertices~$u,v$) completes the proof for the $|\cY_{uv}(i+1)|$ variables.

Finally, for $|\cZ_{uv}(i+1)|$ note that the one-step difference
\begin{equation}\label{def:Zdiff}
\Delta Z := |\cZ_{uv}(i+1)|-|\cZ_{uv}(i)| = \sum_{w\in\cX_{uv}(i)}\indic{uw \in \Gamma_{i+1} \text{ and } vw \in \Gamma_{i+1}} + \sum_{f \in \cY_{uv}(i)}\indic{f \in \Gamma_{i+1}}
\end{equation}
is a sum of independent Bernoulli random variables with
\begin{equation}\label{eq:EZdiff}
\E(\Delta Z) = |\cX_{uv}(i)| \cdot p^2 + |\cY_{uv}(i)| \cdot p \le \csig^2 \cq_i^2 + 2\csig\cq_i\pi_i \le 3 \csig \ll 1 ,
\end{equation}
where we used $|\cX_{uv}(i)|\le \cq_i^2n$ and $|\cY_{uv}(i)|\le 2\cq_i\cpi_i\sqrt{n}$ for the first inequality, 
and $\max\set{\cq_i^2, \cq_i\cpi_i} \le 1$ (see~\eqref{ai1}) and $\csig \ll 1$ for the last two inequalities. 
Inspecting~\eqref{def:Zdiff}, note that $\cG_{\le i}\subseteq \cP_i$ implies $|\cZ_{uv}(i+1)| \le \Delta Z + i (\log n)^9$. 
Applying standard Chernoff bounds, using $\E(\Delta Z) \ll 1$ it readily follows that, say, 
\begin{equation*}
\P\bigpar{|\cZ_{uv}(i+1)| \ge (i+1) (\log n)^9} \le 
\P\bigpar{\Delta Z \ge (\log n)^9} \le n^{-\omega(1)}. 
\end{equation*}
Taking a union bound over all pairs of vertices~$u,v$ completes the proof for the $|\cZ_{uv}(i+1)|$ variables. 
\end{proof}

\begin{remark}\label{rem:codegree}%
If desired, it would not be difficult to establish the better upper bound~$|\cZ_{uv}(i)| \le (\log n)^2$, say
(using the stochastic domination arguments leading to~\eqref{eq:XUT} in \refS{sec:P5}; 
in view~\eqref{def:Zdiff}--\eqref{eq:EZdiff} the main point is that, for~$0 \le i \le I$, the event 
$\cG_{\le i}$ implies $\sum_{0 \le j \le i}(|\cX_{uv}(j)|p^2 + |\cY_{uv}(j)|p) = O(\log n)$). 
This in turn could, e.g., be used to increase the constant~$\beta_0$ slightly  
(as we could then remove~$I=\ceil{n^{\beta}}$ from constraint~\eqref{ai7}). 
\end{remark}

\subsection{Event $\cQ^+_{i+1} \cap\cQ_{i+1}$: number $|\vO_{i+1}(A,B)|$ of open edges between large sets}\label{sec:P3P4} 
Turning to $|\vO_{i+1}(A,B)|$, 
note that one edge~$e \in \Gamma_{i+1}$ can add up to $|Y_{e}(i) \cap \vO_i(A,B)| \le \sum_{w \in e}|N_{\vE_i}(w) \cap (A \cup B)|$ edges 
to $\cC^1_{i+1}(A,B) \subseteq \vO_i(A,B) \setminus \vO_{i+1}(A,B)$, 
which can potentially lead to large edge-effects~$c_{e}$. 
To sidestep such technical difficulties, 
we now introduce the following auxiliary variables for vertex-sets~$A, B \subseteq V$ with~$|A|=|B|$ 
(to avoid clutter we suppress the dependence on $A,B,i$ in parts of our notation): 
\begin{align*}
z &:= \csig^4\cq_i^2|A|, \\
W_1 &:= \set{w \in V : \: |N_{\vE_i}(w)\cap(A\cup B)|\ge z} , \\
W_2 &:=\set{w \in V : \: |N_{\Gamma_{i+1}}(w)\cap (A\cup B)| \ge  z},  \\
\chC^1_{i+1} &:= \set{uv\in\vO_i : \: \text{there is $w \not\in W_1$ s.t. } |\{uw,vw\} \cap \Gamma_{i+1}|=|\{uw,vw\} \cap \vE_i|=1}, \\
\chC^2_{i+1} &:= \set{uv\in\vO_i : \: \text{there is $w \not\in W_2$ s.t. } uw\in\Gamma_{i+1}, \: vw\in\Gamma_{i+1}} , \\
\chC_{i+1} &:= \chC^1_{i+1} \cup \cW_{i+1}.
\end{align*}
Note that $\chC^j_{i+1} \subseteq \cC^j_{i+1}$ for $j \in \{1,2\}$, and that $\chC_{i+1} \subseteq \cC_{i+1}$. 
Furthermore, recalling~$q_i \ge q_I$ (see~\eqref{ai5}), 
using inequality~\eqref{ai7} it is routine to check that~$s_0  \gg 1$ holds, 
that~$|A| \ge s_0$ implies~$z \gg 1$, and  moreover that
\begin{equation}\label{eq:z:frac}
\min_{|A| \ge s_0} z/\sqrt{|A| I} \ge \csig^4\cq_i^2\sqrt{s_0}/\sqrt{I} \gg \csig^6\cq_I^3 \sqrt[4]{n}/\sqrt{I} \gg n^{\tau/2} .
\end{equation}
\begin{lemma}\label{property3}
We have 
$\P(\neg\cQ^+_{i+1})\le n^{-\omega(1)}$. 
\end{lemma}
\begin{proof}
Mimicking the double counting argument from~\eqref{eq:handshake}, 
it follows that the special case $|A|=|B|$ of~$\cQ^+_{i+1}$ implies the event~$\cQ^+_{i+1}$ in full. 
Hence $\neg\cQ^+_{i+1}$ implies that $|\vO_{i+1}(A,B)| \le \cq_{i+1}|A||B|$ fails for some disjoint vertex-sets $A,B \subseteq V$ with~$|A|=|B| \ge s_0$, 
and we shall below estimate the probability of this special case.

Recalling $\chC_{i+1} \subseteq \cC_{i+1}$, 
noting $\vO_{i+1} \subseteq \vO_i \setminus \cC_{i+1} \subseteq  \vO_i \setminus \chC_{i+1}$ we obtain 
\begin{equation}\label{def:XO}
|\vO_{i+1}(A,B)|\le |\vO_{i}(A,B) \setminus \chC_{i+1}| = \sum_{f\in\vO_{i}(A,B)}\indic{f\not\in\chC_{i+1}} =: X .
\end{equation}
To estimate~$\E X$, recall that $\cC^1_{i+1} = \set{f \in \vO_i : \cY_f(i) \cap \Gamma_{i+1} \neq \emptyset}$. 
Note that if the event $\cQ_{f} := \set{(f \times W_1) \cap \Gamma_{i+1} = \emptyset}$ holds, 
then $f \notin \chC^1_{i+1}$ implies $f \notin \cC^1_{i+1}$, 
so that $f \notin \chC_{i+1}$ implies $f \notin \cC_{i+1}= \cC^1_{i+1} \cup \cW_{i+1}$. 
Since~$f \notin \cC^1_{i+1}$ and~$\cQ_{f}$ are both monotone decreasing functions 
of the edge-indicators~$(\indic{e \in \Gamma_{i+1}}, \indic{e \in \cW_{i+1}})_{e \in \vO_i}$, 
using Harris's inequality~\cite{Harris1960} and $\P(\cQ_{f})\ge (1-p)^{2|W_1|}$ it follows that  
\begin{equation*}%\label{eq:PrHarris}
\P(f \notin \cC_{i+1}) \ge \P(f \notin \chC_{i+1} \text{ and } \cQ_{f}) 
\ge \P(f \notin \chC_{i+1}) \P(\cQ_{f})
\ge \P(f \notin \chC_{i+1}) \cdot (1-p)^{2|W_1|} .
\end{equation*}
Note that $\cG_{\le i}$ and~$i < I$ imply $|N_{\vE_i}(u)\cap  N_{\vE_i}(v)| = |\cZ_{uv}(i)| \le I(\log n)^9 =: y$  when~$u\neq v$, 
and that~\eqref{eq:z:frac} implies~$z\gg\sqrt{|A\cup B|y}$. 
Using the definition of~$W_1$ and \refL{lem:overlap} (with $\cI = W_1$, $U=A\cup B$ and $U_w = N_{\vE_i}(w)\cap U$),  
we infer $|W_1|\le 2|A\cup B|/z = 4/(\csig^4\cq_i^2) \le \cq_i\csig\sqrt{n}$ by~\eqref{ai7}, say.  
Similar to~\eqref{eq:Xuv}, using \refL{lemma1},  $|\vO_{i}(A,B)| \le \cq_i |A||B|$, $p|W_1| \le \cq_i\csig^2 \ll 1$ and $\cq_{i}\cq_{i+1} \sim q_i^2$ (see~\eqref{ai5}) it is routine to deduce that 
\begin{equation}\label{eq:OiAB:LT:E}
\begin{split}
\E X &\le |\vO_{i}(A,B)| \cdot (\cq_{i+1}/\cq_i-\csig^{3/2}\cq_i) \cdot (1-p)^{-2|W_1|} \le |A||B| \cdot (\cq_{i+1}-\csig^{3/2}\cq_i^2/2) .
\end{split}
\end{equation}
Gearing up to apply \refT{thm:BDI}, 
we now bound $\lambda \le p \sum_{e \in \vO_i} c_e^2  + q_i\sum_{e \in \vO_i} \hc_e^2$. 
Noting $\chC_{i+1} \subseteq \cC_{i+1}$, as usual we have edge-effect $c_e \le |\cY_e(i)\cap \vO_i(A,B)|$ and stabilization-effect $\hat{c}_e \le \indic{e\in\vO_i(A,B)}$. 
Here the definition of~$\chC_{i+1}$ allows us to improve the simple worst case estimate~$c_e\le |\cY_e(i)|$. 
Indeed, inspecting the corresponding argument for $|N_{\vO_{i+1}}(v)|$ from \refL{lem:Pro1}, 
we see that the edge-effect~$c_e$ (an upper bound on how much~$X$ changes if we alter whether~$e \in \Gamma_{i+1}$ or~$e \not\in \Gamma_{i+1}$) 
is at most the number of changes~to 
\begin{equation}\label{eq:chCO}
\begin{split}
\chC^1_{i+1} \cap \vO_i(A,B) = \big\{uv\in\vO_i(A,B) : \:  \text{there is $w \not\in W_1$ s.t.\ either }& uw \in \Gamma_{i+1}, \: vw \in \vE_i \\
& \text{or } vw \in \Gamma_{i+1}, \: uw \in \vE_i\bigr\}. 
\end{split}
\end{equation}
Since any~$w \not\in W_1$ has at most~$z$ neighbours in $A \cup B$ via $\vE_i$--edges,  
we infer that~$c_e \le 2z$ (the factor of two takes into account that each vertex of~$e$ could potentially play the role of~$w$ in~\eqref{eq:chCO} above). 
Similar to~\eqref{eq:SL1} and~\eqref{eq:SL2}, 
using \refL{summinglemma}, $\csig \cpi_i \le \sqrt{\csig} \ll 1$ (see~\eqref{ai1}), 
and $|\vO_i(A,B)| \le \cq_i|A||B|$ it follows that 
\[
p \sum_{e \in \vO_i} c_e^2 
\le \csig/\sqrt{n} \cdot 2z \cdot 2\cq_i\cpi_i\sqrt{n}\cdot |\vO_i(A,B)| 
\ll z\cq_i |\vO_i(A,B)| 
\le z\cq_i^2|A||B| .
\]
Furthermore, using $z \ge 1$ we obtain $q_i\sum \hat{c}_e^2\le q_i|\vO_i(A,B)| \le z\cq_i |\vO_i(A,B)|  \le z\cq_i^2|A||B|$. 
Noting that~$X$ %defined in~\eqref{def:XO} 
is decreasing, using \refT{thm:BDI} 
and the $\lambda$--bound~\eqref{eq:l} 
it follows that 
\begin{equation}\label{eq:OiAB:LT}
\begin{split}
\P(|\vO_{i+1}(A,B)| \ge \cq_{i+1}|A||B|) & \le \P(X \ge \E X + \csig^{3/2}\cq_i^2|A||B|/2) \\
& \le \exp\lrpar{-\frac{\bigpar{\csig^{3/2}\cq_i^2|A||B|/2}^2}{2 \cdot 2z\cq_i^2|A||B|}}
= \exp\lrpar{-\frac{\csig^{3}\cq_{i}^2|A||B|}{16z}} \le n^{-\omega(|B|)} ,
\end{split}
\end{equation}
where for the last inequality we used $z = \csig^4\cq_i^2|A|$ and $\csig^{-1} \gg \log n$. 
Finally, taking a union bound over all disjoint vertex-sets $A,B \subseteq V$ with $|A|=|B| \ge s_0$ completes the proof (as discussed). 
\end{proof}

For the `relative error'~$\tau_i$ used in the event~$\cQ_{i}$, see~\eqref{def:tau},  
we now record the following convenient bounds:
\begin{equation}
\label{eq:taui}
1 \ge \tau_i \ge \tau_I 
%\ge 1- \delta/2 \cdot \pi_i/\pi_I 
= 1- \delta/2 \ge 1/2 \qquad \text{ for all $0 \le i \le I$.}
\end{equation}
\begin{lemma}\label{cPro4}
We have $\P(\neg \cQ_{i+1} \cap \cN_{i+1} \cap \cP_{i+1})\le n^{-\omega(1)}$.
\end{lemma}

The proof strategy is to estimate the different contributions to 
$\vO_{i+1} = \vO_i \setminus (\Gamma_{i+1}\cup\cC_{i+1} \cup \cC^2_{i+1})$ separately 
(here~$\cQ_i^+$ will be crucial for bounding some of the large edge-effects ignored in \refL{property3}). 

\begin{claim}\label{claimP4}
Let $\cQ_{A,B}$ be the event that the following bounds hold: 
\begin{align*}
X_1 &:=\bigabs{\vO_i(A,B)\setminus \chC_{i+1}} \in \bigl[|\vO_i(A,B)| \cdot (\cq_{i+1}/\cq_i-4\csig^{3/2}\cq_{i}) , \: |\vO_i(A,B)| \cdot \cq_{i+1}/\cq_i\bigr],\\
X_2 &:=\bigabs{\vO_i(A,B)\cap\chC^2_{i+1}} \le |\vO_i(A,B)| \cdot 2\csig^2\cq_i ,\\
X_3 &:=|\vO_i(A,B)\cap\Gamma_{i+1}|\le |\vO_i(A,B)| \cdot 2\csig^2\cq_i,\\
X_4 &:=\bigabs{\vO_i(A,B)\cap (\cC_{i+1} \cup \cC^2_{i+1}) \setminus (\chC_{i+1} \cup \chC^2_{i+1})}\le 36\csig\cq_i^2\sqrt{n}|A|.
\end{align*}
Then $\P(\neg \cQ_{A,B}\cap\cN_{i+1}\cap\cP_{i+1})\le n^{-\omega(s)}$ for all vertex-sets $(A,B) \in \fS_s$. 
\end{claim}

Before giving the proof, we first show that \refCl{claimP4} implies \refL{cPro4}. 
Using a union bound argument (to account for the $|\fS_s| \le n^{2s}$ vertex-sets $(A,B) \in \fS_s$), 
it is enough to show that $\cQ_{A,B} \cap \cG_{\le i}$ implies $\tau_{i+1}\cq_{i+1}|\vO_0(A,B)| \le |\vO_{i+1}(A,B)| \le \cq_{i+1}|\vO_0(A,B)|$. 
By definition of $\vO_{i+1}(A,B)$ we have 
\begin{equation*}
X_{1}-X_{2}-X_3-X_{4} \le |\vO_{i+1}(A,B)| \le X_1 .
\end{equation*}
Combining~$\cQ_{A,B}$ with the fact that $|\vO_i(A,B)| \le \cq_i |\vO_0(A,B)|$ 
by~$\cG_{\le i} \subseteq \cQ_i$, we readily infer the upper bound $|\vO_{i+1}(A,B)| \le \cq_{i+1}|\vO_0(A,B)|$. 
Turning to the lower bound, using~$\cQ_{A,B}$ it follows that 
\begin{equation*}
\begin{split}
X_{1}-X_{2}-X_3-X_{4} &\ge |\vO_i(A,B)| \cdot \bigpar{\cq_{i+1}/\cq_i-8\csig^{3/2}\cq_{i}}-36\csig\cq_i^2\sqrt{n}|A|\\
&\ge \Bigpar{\ctau_i\cq_i\bigpar{\cq_{i+1}/\cq_i-8\csig^{3/2}\cq_{i}}-\frac{36\csig\cq_i^2}{\gamma C\sqrt{\log n}}} \cdot |\vO_0(A,B)|\\
&\ge \Bigpar{\ctau_i-\frac{45\csig\cq_i}{\gamma C\sqrt{\log n}}} \cdot  \cq_{i+1} |\vO_0(A,B)| 
\ge \ctau_{i+1} \cdot \cq_{i+1}|\vO_0(A,B)| ,
\end{split}
\end{equation*}
where for the second inequality we used $|\vO_i(A,B)|\ge \ctau_i\cq_i|\vO_0(A,B)|$ (by~$\cG_{\le i} \subseteq \cQ_i$) and $|\vO_0(A,B)|\ge\gamma|A||B|\ge\gamma C\sqrt{\log n}\cdot \sqrt{n}|A|$, 
for the third inequality we used $\ctau_i \le 1$ (see~\eqref{eq:taui}), $\sigma^{1/2} \ll 1/\sqrt{\log n}$,  and $\cq_i \sim \cq_{i+1}$ (see~\eqref{ai5}), 
and for the last inequality we used $\sqrt{\log n} \sim \sqrt{\log(I \csig)/\beta} \sim \cpi_I/\sqrt{\beta}$ (see~\eqref{ai2}), $\gamma C/\sqrt{\beta} \ge D_0 /\delta^2 \ge 91/\delta$ (by assumption and~\eqref{def:D0beta0}) and $\ctau_{i}-\delta\sigma\cq_i/\cpi_I=\ctau_{i+1}$ (see~\eqref{def:tau}). 
This completes the proof of \refL{cPro4} (assuming \refCl{claimP4}).

\begin{proof}[Proof of \refCl{claimP4}]
We start with $X_1=|\vO_i(A,B)\setminus\chC_{i+1}|$. 
Since $s \ge s_0$, the upper tail argument for~$X=X_1$ defined in~\eqref{def:XO} 
carries over from \refL{property3}, 
with $\E X_{1} \le  |\vO_i(A,B)| (\cq_{i+1}/\cq_i-\csig^{3/2}\cq_i/2)$ and $\lambda \le 2z\cq_i |\vO_i(A,B)|$, say.  
In particular, 
noting that here $|\vO_i(A,B)|\ge\ctau_i\cq_i|\vO_0(A,B)|\ge \gamma\ctau_i\cq_i|A||B|$, 
an application of \refT{thm:BDI} along the lines of~\eqref{eq:OiAB:LT} gives 
\begin{equation}\label{eq:X1:LT}
\begin{split}
\P(X_1 \ge |\vO_i(A,B)| \cq_{i+1}/\cq_i) 
& \le \exp\lrpar{-\frac{\bigpar{\csig^{3/2}\cq_i|\vO_i(A,B)|/2}^2}{2 \cdot 2z\cq_i |\vO_i(A,B)|}} 
 \le \exp\lrpar{-\frac{\gamma\ctau_i\csig^{3}\cq_{i}^2|A||B|}{16z}} \le n^{-\omega(s)} ,
\end{split}
\end{equation}
where for the last inequality we used $z =\csig^4\cq_i^2|A|$, $\tau_i \ge 1/2$ (see~\eqref{eq:taui}), $\gamma \csig^{-1} \gg \log n$ and $|B|=s$.  
For the lower tail of~$X_1$ we proceed similarly. 
Since $\chC_{i+1} \subseteq \cC_{i+1}$, using \refL{lemma1} we obtain  
\[\E X_{1}=\sum_{e\in\vO_i(A,B)}\P(e\not\in \chC_{i+1})
\ge \sum_{e\in\vO_i(A,B)} \P(e\not\in\cC_{i+1})
\ge |\vO_i(A,B)| \cdot (\cq_{i+1}/\cq_i-3\csig^{3/2}\cq_i).\] 
Furthermore, the edge-effect and stabilization-effect estimates from the proof of \refL{property3} again carry over, 
giving $\lambda \le 2z\cq_i |\vO_i(A,B)|$ and $\max_{e \in \vO_i} \max\{c_e,\hc_e\} \le 2z$, say.  
Applying inequality~\eqref{eq:BDI} of \refR{rem:BDI:LT} (with~$C=2z$), 
it follows similarly to~\eqref{eq:X1:LT} that 
\begin{equation}\label{eq:X1}
\begin{split}
\P\bigpar{X_{1}\le |\vO_i(A,B)| (\cq_{i+1}/\cq_i-4\csig^{3/2}\cq_{i})} & \le 
\P\bigpar{X_{1}\le\E X_{1} -\csig^{3/2}\cq_{i}|\vO_i(A,B)|}\\
&\le \exp\lrpar{-\frac{\bigpar{\csig^{3/2}\cq_{i}|\vO_i(A,B)|}^2}{
	2\bigpar{2z\cq_i|\vO_i(A,B)|+2z \cdot \csig^{3/2}\cq_{i}|\vO_i(A,B)|}}}\\
&
\le \exp\lrpar{-\frac{\gamma\ctau_i\csig^{3}\cq_i^2|A||B|}{8z}} 
\le n^{-\omega(s)} .
\end{split}
\end{equation}

Turning to $X_2=|\vO_i(A,B)\cap\chC^2_{i+1}|$, note that by construction of $\chC^2_{i+1}$ we have 
\begin{equation}\label{N1}
X_2 = \sum_{e \in \vO_i(A,B)} \indic{e \in \chC^2_{i+1}} \le\sum_{ab \in O_i(A,B)} \sum_{w \in V \setminus W_2}\indic{\set{wa,wb}\subseteq\Gamma_{i+1}} =: X_2^+ .
\end{equation}
Gearing up to apply \refT{thm:UT} to~$X_2^+$, in view of $\Gamma_{i+1} \subseteq \vO_i$ we define 
\begin{align*}
\cI & := \bigset{\set{wa,wb}\subseteq \vO_i: \: ab \in \vO_i(A,B), \: w\in V, \: |\set{a,b,w}|=3} , \\
\cK & :=\set{\set{wa,wb}\in\cI: \: w \not\in W_2, \: \set{wa,wb} \subseteq \Gamma_{i+1}} .
\end{align*}
Since $p^2 \cdot |\cX_{ab}(i)|\le \csig^2\cq_i^2 \le \csig^2\cq_i$ by $\cG_{\le i} \subseteq \cP_i$ and $\cq_i \le 1$ (see~\eqref{ai1}), we obtain 
\begin{equation*}
\begin{split}
\sum_{\alpha\in\cI}\E \indic{\alpha \subseteq \Gamma_{i+1}} 
%= p^2 \cdot |\cI| 
&= p^2 \sum_{ab \in \vO_i(A,B)} \sum_{v \in V}\indic{\set{va,vb} \subseteq \vO_i} 
= p^2 \sum_{ab \in \vO_i(A,B)} |\cX_{ab}(i)| \le \csig^2\cq_i \cdot |\vO_i(A,B)|=:\mu.
\end{split}
\end{equation*}
Furthermore, since $\cK$ only contains edge-pairs $\set{wa,wb}$ with $\set{a,b} \subseteq  N_{\Gamma_{i+1}}(w) \cap (A \cup B)$  
where the `central vertex'~$w$ satisfies $w  \not\in W_2$ and thus $|N_{\Gamma_{i+1}}(w)\cap (A\cup B)| \le z$, 
for all $\beta \in \cK$ we see that 
\[
|\set{\alpha \in \cK: \: \alpha \cap \beta \neq \emptyset}| \le \sum_{f \in \beta} |\set{\alpha \in \cK: \: f \in \alpha}| \le \sum_{f \in \beta} \sum_{v \in f \setminus W_2}|N_{\Gamma_{i+1}}(v)\cap (A\cup B)| \le 2 \cdot 2 \cdot z .
\]
It follows that $X_2^+ = \sum_{\alpha \in \cK}\indic{\alpha \subseteq \Gamma_{i+1}} \le Z_{4z}$, where $Z_{4z}$ is defined as in \refT{thm:UT}. 
Applying first~\eqref{N1} and then inequality~\eqref{eq:C} with $C=4z$, using 
$|\vO_i(A,B)| \ge \gamma\ctau_i\cq_i |A||B|$ it follows similarly to~\eqref{eq:X1:LT} that 
\begin{equation}\label{eq:X2}
\begin{split}
\P( X_2\ge 2\csig^2\cq_i|\vO_i(A,B)|) &\le\P(Z_{4z}\ge 2\mu)
\le \exp\Bigpar{-\frac{\mu^2}{2 \cdot 4z \cdot 2\mu}}  
\le \exp\Bigpar{-\frac{\gamma\ctau_i\csig^{2}\cq_i^2|A||B|}{16z}} \le n^{-\omega(s)} .
\end{split}
\end{equation}

We next turn to $X_3=|\vO_i(A,B)\cap \Gamma_{i+1}|$, 
which is a sum of independent Bernoulli random variables with 
$\E X_3 =|\vO_i(A,B)| \cdot p \ll \csig^2\cq_i |\vO_i(A,B)| =: t$, as~$\cq_i \sqrt{n} \ge n^{\tau}$ by~\eqref{ai7}. 
Applying standard Chernoff bounds, using $|\vO_i(A,B)| \ge \gamma\ctau_i\cq_i |A||B|$ 
and $z \ge 1$ it follows by comparison with the last inequality of~\eqref{eq:X2} that 
\begin{equation}\label{eq:X3}
\P(X_3 \ge 2\csig^2\cq_i|\vO_i(A,B)|) \le \P(X_3 \ge \E X_3 +t) 
\le \exp\Bigpar{-\frac{t^2}{2 \cdot 2t}} 
\le \exp\Bigpar{-\frac{\gamma\ctau_i\csig^2\cq_i^2 |A||B|}{4}} \le n^{-\omega(s)} .
\end{equation}

Finally, $X_4$~is a more difficult variable: 
assuming that $\cN_{i+1} \cap \cP_{i+1} \cap \cG_{\le i}$ holds, 
we shall bound~$X_4$ by \emph{deterministic} counting arguments 
(here the edge-effects can potentially be fairly large, so concentration inequalities seem less effective).  
Noting $\cC_{i+1}\setminus\chC_{i+1}=\cC^1_{i+1}\setminus\chC^1_{i+1}$, similarly to~\eqref{N1} we obtain 
\begin{equation}\label{equationX4}
\begin{split}
X_4 
&\le \sum_{e\in\vO_i(A,B)}\indic{e\in\cC^1_{i+1}\setminus\chC^1_{i+1}}+
\sum_{e\in\vO_i(A,B)}\indic{e\in\cC^2_{i+1}\setminus\chC^2_{i+1}}\\
&\le \sum_{w\in W_1}\Bigpar{|\vO_i( N_{\Gamma_{i+1}}(w)\cap A, \:  N_{\vE_i}(w)\cap B)|+|\vO_i( N_{\Gamma_{i+1}}(w)\cap B, \:  N_{\vE_i}(w)\cap A)|}\\
&\qquad  +\sum_{w \in W_2}|\vO_i( N_{\Gamma_{i+1}}(w)\cap A, \:  N_{\Gamma_{i+1}}(w)\cap B)|.
\end{split}
\end{equation}
Using the upper bound estimate from $\cG_{\le i} \subseteq \cQ^+_i$ when $\min\{ |N_{\Gamma_{i+1}}(v)\cap A|,|N_{\vE_{i}}(v)\cap B|\} \ge z$ holds 
(note that $z = \csig^4 q_i^2 s \ge s_0$), and a trivial estimate otherwise, 
it follows that 
\begin{equation}\label{equation:X41a}
\begin{split}
& |\vO_i( N_{\Gamma_{i+1}}(w)\cap A,  \: N_{\vE_i}(w)\cap B)| \\
& \qquad \le \cq_i |N_{\Gamma_{i+1}}(w)\cap A||N_{\vE_i}(w)\cap B| + z\max\{ |N_{\Gamma_{i+1}}(w)\cap A|,|N_{\vE_{i}}(w)\cap B|\}\\
& \qquad \le \bigpar{\cq_i |N_{\Gamma_{i+1}}(w)| + z} \cdot |N_{\vE_i\cup\Gamma_{i+1}}(w)\cap(A\cup B)| . 
\end{split}
\end{equation}
With an eye on~\eqref{equationX4}, we note that an analogous estimate also holds when we reverse the role of~$A$ and~$B$ in~\eqref{equation:X41a}. 
Furthermore, $\cq_i|N_{\Gamma_{i+1}}(w)|\le 2\csig \cq_i^2\sqrt{n}$ by~$\cN_{i+1}$, and $z = \sigma^4q_i^2 s = O(\sigma^3q_i^2\sqrt{n}) \ll \csig\cq_i^2\sqrt{n}$. 
Recalling $\vE_i\cup\Gamma_{i+1} = \vE_{i+1}$, 
observe that~$\cP_{i+1}$ and~$i+1 \le I$ imply $|N_{\vE_i\cup\Gamma_{i+1}}(u) \cap  N_{\vE_i\cup\Gamma_{i+1}}(v)| = |\cZ_{uv}(i+1)| \le I (\log n)^9 =: y$ when~$u \neq v$, 
and that~\eqref{eq:z:frac} implies~$z \gg \sqrt{|A \cup B|y}$ (as $|A|= s \ge s_0$). 
Using the definition of~$W_1$ and \refL{lem:overlap} (with $\cI = W_1$, $U=A\cup B$ and $U_w = N_{\vE_i\cup\Gamma_{i+1}}(w) \cap U$), it follows that 
\begin{equation}\label{equation:X411}
\begin{split}
&\sum_{w\in W_1}\Bigpar{|\vO_i( N_{\Gamma_{i+1}}(w)\cap A, \:  N_{\vE_i}(w)\cap B)|+|\vO_i( N_{\Gamma_{i+1}}(w)\cap B, \:  N_{\vE_i}(w)\cap A)|}\\
&\qquad \le 2 \cdot 3\csig\cq_i^2\sqrt{n} \cdot \sum_{w\in W_1} |N_{\vE_i\cup\Gamma_{i+1}}(w)\cap(A\cup B)| 
\le 2 \cdot 3\csig\cq_i^2\sqrt{n} \cdot 2|A\cup B|\le 24\csig\cq_i^2\sqrt{n}|A|. 
\end{split}
\end{equation}
Proceeding analogously to~\eqref{equation:X41a}--\eqref{equation:X411}, using the definition of~$W_2$ and \refL{lem:overlap} we similarly obtain
\begin{equation}\label{equation:X42}
\begin{split}
&\sum_{w\in W_2}|\vO_i( N_{\Gamma_{i+1}}(w)\cap A, \: N_{\Gamma_{i+1}}(w)\cap B)| \\
&\qquad \le 3\csig\cq_i^2\sqrt{n} \cdot \sum_{w\in W_2} |N_{\Gamma_{i+1}}(w)\cap(A\cup B)| 
\le 3\csig\cq_i^2\sqrt{n} \cdot 2|A\cup B|\le 12\csig\cq_i^2\sqrt{n}|A| .
\end{split}
\end{equation}
To sum up, inserting the bounds~\eqref{equation:X411}--\eqref{equation:X42} into \eqref{equationX4}, we showed that $\cR_{i+1} \cap \cP_{i+1} \cap \cG_{\le i}$ implies $X_4\le 36\csig\cq_i^2\sqrt{n}|A|$. 
This completes the proof together with the probability estimates~\eqref{eq:X1:LT}, \eqref{eq:X1}, \eqref{eq:X2}, and~\eqref{eq:X3}.  
\end{proof}

\begin{remark}\label{rem:fS_larger}%
If desired, it would not be difficult to extend the event~$\cQ_{i}$ % defined in~\eqref{def:Qi}
to larger vertex-sets $(A,B) \in \fS_{\ge s} := \bigcup_{s \le r \le n} \fS_{r}$ 
(the above arguments all carry over, except for the modified 
bound~$X_4\le 3 \cdot \max_w(\cq_i |N_{\Gamma_{i+1}}(w)| + z) \cdot 2|A \cup B| \le 36\csig\cq_i^2 \max\{\sqrt{n}, \csig^3|B|\}|A|$, 
which is still strong enough to deduce \refL{cPro4}). 
This in turn could, e.g., be used to also extend the event~$\cT_I$ to $(A,B) \in  \fS_{\ge s}$ 
(the proofs in \refS{sec:P5} then carry over). 
\end{remark}

\begin{remark}\label{rem:OiTIdirect}%
Under a mild extra assumption such as $|\vO_0| \ge \csig n$, say, it would not be difficult to 
add two-sided bounds for the total number of open edges~$|O_i|$ and edges~$|\vF_I|$ to the events~$\cQ_{i}$ and $\cT_I$. 
For example, much simpler variants of the above arguments then imply $\tau_i q_i |\vO_{0}| \le |O_{i}| \le q_i |\vO_{0}|$ 
(by directly estimating $|\vO_{i} \setminus \cC_{i+1}|-|\Gamma_{i+1}|-|\cC^2_{i+1}| \le |\vO_{i+1}| \le |\vO_{i} \setminus \cC_{i+1}|$, 
without using~$\chC_{i+1}$ or $\chC^2_{i+1}$, nor a union bound over all vertex-sets), %~$(A,B) \in \fS_s$), 
which in turn gives $|\vF_I| = (1 \pm \delta) \rho |\vO_{0}|$ by straightforward variants of the proofs in \refS{sec:P5}. 
\end{remark}

\subsection{Event $\cT_I$: number $|\vF_I(A,B)|$ of edges between large sets}\label{sec:P5} 
For $|\vF_I(A,B)|$ it is convenient to think of the entire 
nibble construction as one evolving random process. 
Thus, in contrast to previous sections, in \refL{P5} and \refCl{claimP5} below we shall \emph{not} tacitly condition on~$\cF_i$. 
\begin{lemma}\label{P5}
We have 
$\P(\neg\cT_I \cap \cG_{\le I})\le n^{-\omega(1)}$. 
\end{lemma}

Since $\vF_I = \bigcup_{0 \le i < I}(\vF_{i+1}\setminus \vF_{i})$ forms a partition, 
the proof strategy is to estimate the two contributions to $\vF_{i+1}\setminus \vF_{i} = \Gamma_{i+1} \setminus E(\cD_{i+1})$ separately 
(here the deleted edges $E(\cD_{i+1})$ will have negligible impact). 

\begin{claim}\label{claimP5}
Let $\cT_{A,B}$ be the event that the following bounds hold: 
\begin{align*}
X &:=\sum_{0\le i <I} |\vO_i(A,B) \cap \Gamma_{i+1}|   \in \bigl[(1-\delta/2)\mu^-, \: (1+\delta/2)\mu^+\bigr], \\
Y &:= \sum_{0\le i <I}\bigl|\vO_{i}(A,B) \cap E(\cD_{i+1})|  \le \delta^2 \mu^-/9,
\end{align*}
where $\mu^+:=\sum_{0\le i< I}\floor{\cq_i |\vO_0(A,B)|}p$ and $\mu^-:=\sum_{0\le i <I}\ceil{\ctau_i\cq_i |\vO_0(A,B)|}p$. 
Then $\P(\neg\cT_{A,B}\cap \cG_{\le I})\le 3 n^{-3s}$ 
for all vertex-sets $(A,B) \in \fS_s$. 
\end{claim}

Before giving the proof, we first show that \refCl{claimP5} implies \refL{P5}. 
Using a union bound argument (to account for the $|\fS_s| \le n^{2s}$ vertex-sets $(A,B) \in \fS_s$), 
it is enough to show that $\cT_{A,B}$ implies $|\vF_I(A,B)| = (1\pm \delta)\rho|\vO_0(A,B)|$. 
Since all the $(\Gamma_{i+1})_{0 \le i < I}$ are edge-disjoint, by the recursive definition~\eqref{def:Ti} of~$\vF_I$ we have 
\begin{equation}\label{eq:FIO}
X-Y\le |\vF_I(A,B)|\le X .
\end{equation}
Noting $\mu^-\ge \ctau_I\mu^+ = (1-\delta/2)\mu^+$ (see~\eqref{eq:taui}), it follows that $\cT_{A,B}$ implies $X\le (1+\delta/2)\mu^+$ and 
\[
X-Y\ge \bigpar{1-\delta/2-\delta^2/9} \cdot \mu^-\ge (1-\delta+\delta^2/8)\mu^+. 
\] 
It thus suffices to show that $\mu^+ \sim \rho |\vO_0(A,B)|$, 
where $\rho=\sqrt{\beta (\log n)/n}$. 
But this is routine: indeed, since $\cq_i|\vO_0(A,B)|\ge\cq_i \cdot \gamma s^2 \gg \cq_i n \gg \sqrt{n}$ by~\eqref{ai7}, 
and $\cpi_I \sim \sqrt{\log(I \csig)} \sim \sqrt{\beta\log n}$ by~\eqref{ai2}, 
using the definition~\eqref{def:pii} of~$\cpi_I$ we readily infer 
\begin{equation}\label{eq:mup}
\begin{split}
\mu^+ &= \sum_{0\le i <I}(\cq_i|\vO_0(A,B)|\pm 1)p  \sim  \sum_{0\le i <I}\sigma\cq_i/\sqrt{n} \cdot |\vO_0(A,B)|\\
&= (\cpi_I-\csig)/\sqrt{n} \cdot |\vO_0(A,B)|  \sim \rho |\vO_0(A,B)| ,
\end{split}
\end{equation}
completing the proof of \refL{P5} (assuming \refCl{claimP5}).

\begin{proof}[Proof of \refCl{claimP5}]
We start with $X=\sum_{0\le i< I}|\vO_i(A,B) \cap \Gamma_{i+1}|$. 
Define 
\begin{equation*}
X_{i+1}^+ := \indic{\cG_i}\sum_{e\in\vO_i(A,B)}\indic{e\in\Gamma_{i+1}} 
 \qquad \text{and} \qquad 
X^+:=\sum_{0\le i< I}X_{i+1}^+ .
\end{equation*}
Note that $X=X^+$ when $\cG_{\le I}=\bigcap_{0\le i\le I}\cG_i$ holds. 
Let $Z_{i+1}^+ \eqd \Bin(\floor{\cq_{i} |\vO_0(A,B)|}, \: p)$ be independent random variables (where $\eqd$ means equality in distribution, as usual). 
Since the $\cF_i$-measurable event $\cG_i \subseteq \cQ_i$ implies $|\vO_i(A,B)| \le \cq_{i} |\vO_0(A,B)|$, 
it is easy to see that $\P(X_{i+1}^+\ge t \mid \cF_i)\le \P(Z_{i+1}^+\ge t)$ for $t\in\RR$. 
Setting 
\begin{equation}\label{eq:Zplus}
Z^+:= \sum_{0\le i< I}Z_{i+1}^+ \eqd \Bin\Bigpar{\sum_{0\le i< I}\floor{\cq_i |\vO_0(A,B)|}, \: p},
\end{equation}
a standard stochastic domination argument then shows $\P(X^+\ge t)\le \P(Z^+\ge t)$ for $t\in\RR$, 
so that  
\begin{equation}\label{eq:XUT}
\P\xpar{X\ge t \text{ and } \cG_{\le I}}\le \P\xpar{X^+\ge t}\le\P\xpar{Z^+\ge t} .
\end{equation}
Since $\cG_i$ also implies $|\vO_i(A,B)| \ge \ctau_i \cq_{i} |\vO_0(A,B)|$, 
an analogous argument gives 
\begin{equation}\label{eq:XLT}
\P\xpar{X\le t \text{ and } \cG_{\le I}} \le\P\xpar{Z^-\le t} % ,
\qquad \text{with} \qquad 
Z^- \eqd \Bin\Bigpar{\sum_{0\le i <I}\ceil{\ctau_i\cq_i |\vO_0(A,B)|}, \: p} .
\end{equation}
Combining $\mu^-\ge \ctau_I\mu^+\ge \mu^+/2$ (see~\eqref{eq:taui}) and~\eqref{eq:mup} with 
$|\vO_0(A,B)|\ge\gamma s^2$, using $\delta^2\sqrt{\beta}\gamma \cdot C\ge D_0 = 108$ (by assumption and~\eqref{def:D0beta0}) 
we have 
\begin{equation}\label{boundmu}
\delta^2\min\{\mu^-,\mu^+\}
\ge \tfrac{\delta^2}{2}\mu^+
\ge \tfrac{\delta^2}{3}\rho |\vO_0(A,B)|\ge
\tfrac{\delta^2}{3}\sqrt{\beta (\log n)/n} \cdot \gamma C\sqrt{ n\log n}\cdot s \ge 36 s\log n.
\end{equation}
Using \eqref{eq:Zplus}--\eqref{eq:XLT} and $\E Z^{\pm}=\mu^{\pm}$, by standard Chernoff bounds (see, e.g., \refR{rem:chernoff}) we obtain, say, 
\begin{equation}\label{eq:XUTLYmu}
\begin{split}
\P\bigpar{X \not\in\bigl[(1-\delta/2)\mu^-, \: (1+\delta/2)\mu^+\bigr] \text{ and } \cG_{\le I}} & \le \P\bigpar{Z^-\le (1-\delta/2)\mu^-} + \P\bigpar{Z^+\ge (1+\delta/2)\mu^+} \\
& \le \exp\bigpar{-\delta^2\mu^-/8} + \exp\bigpar{-\delta^2\mu^+/12} \le 2 n^{-3s} . 
\end{split}
\end{equation}

Finally, turning to $Y=\sum_{0\le i< I}|\vO_{i}(A,B)\cap E(\cD_{i+1})|$, for brevity we define 
\begin{equation*}
Y_{i+1} :=|\vO_{i}(A,B)\cap E(\cD_{i+1})|
 \qquad \text{and} \qquad 
y := \delta^2 \mu^-/9 .
\end{equation*}
Note that $Y=\sum_{0\le i < I}Y_{i+1}$ and $Y_{i+1} \in \NN$. 
Since $\cG_{\le i}=\bigcap_{0\le j\le i}\cG_j$, 
a union bound argument gives
\begin{equation}\label{eq:sumbound}
\begin{split}
\P\bigpar{Y\ge \delta^2 \mu^-/9 \text{ and } \cG_{\le I}} &\le
\sum_{\substack{(y_1,\dots,y_I)\in\mathbb{N}^I\\
		\sum_{1 \le i \le I}y_i=\ceil{y}}}
\P\Bigpar{\bigcap_{0\le i< I}\bigpar{Y_{i+1}\ge y_{i+1} \text{ and } \cG_{\le i+1}}}\\
&\le\sum_{\substack{(y_1,\dots,y_I)\in\mathbb{N}^I\\
		\sum_{0 \le i < I}y_{i+1}=\ceil{y}}}
\prod_{0\le i< I}\P\Bigpar{Y_{i+1}\ge y_{i+1} \mathrel{\Big|}
	\bigcap_{0\le j< i}\bigpar{Y_{j+1}\ge y_{j+1} \text{ and } \cG_{\le j+1}}} .
\end{split}
\end{equation}
Gearing up to apply \refT{thm:UT} to~$Y_{i+1}$, with an eye on 
$\cD_{i+1} \subseteq \cB_{i+1}$ and $\vF_i \subseteq \vE_{i}$ (see \refS{sec:nibbledetails}) 
we define
\begin{align*}
\cI & := \bigset{\set{wu,wv}\subseteq \vO_i: \: uv\in \vE_{i}, \: |\set{u,v,w}|=3, \: \set{wu, wv} \cap \vO_i(A,B) \neq\emptyset} \\
& \qquad \qquad \cup \: \bigset{\set{uv,vw,wu}\subseteq \vO_i: \: |\set{u,v,w}|=3, \: \set{uv, vw, wu} \cap \vO_i(A,B) \neq \emptyset}.
\end{align*}
Since each edge-set $\alpha \in \cI$ contains at least one edge from $\vO_i(A,B)$, 
when the $\cF_i$-measurable event $\cG_{\le i}$ holds we infer by the usual reasoning (using, e.g., $\cP_i \cap \cQ_i$ and $\max\{\cpi_i\cq_i, \cq_i^2\} \le 1$) that  
\begin{equation*}
\begin{split}
\sum_{\alpha \in \cI} \E(\indic{\alpha \subseteq \Gamma_{i+1}} \mid \cF_{i}) & \le \sum_{e \in \vO_i(A,B)} \sum_{\alpha \in \cI: e \in \alpha} p^{|\alpha|} \le \sum_{e \in \vO_i(A,B)}\Bigpar{|\cY_e(i)| \cdot p^2 + |\cX_e(i)| \cdot p^3 }\\
& \le \cq_i|\vO_0(A,B)| \cdot \bigpar{2\cpi_i\cq_i\sqrt{n} \cdot p^2 + \cq_i^2n \cdot p^3} \le 3\csig \cdot \cq_i|\vO_0(A,B)|p = : \mu_{i+1}^* .
\end{split}
\end{equation*}
Since~$\cD_{i+1}$ is a collection of edge-disjoint elements of $\cB_{i+1}$
(and thus $\set{\alpha \in \cD_{i+1}: \alpha \cap \beta \neq \emptyset} = \{\beta\}$ for all~$\beta \in \cD_{i+1}$),  
using $E(\cD_{i+1}) = \bigcup_{\alpha\in\cD_{i+1}}\alpha \subseteq \Gamma_{i+1} \subseteq \vO_i$, $|\alpha| \le 3$ and $\vF_i \subseteq \vE_{i}$ 
it is not difficult to check that 
\[
Y_{i+1} = 
\sum_{\alpha \in \cD_{i+1}} |\alpha \cap \vO_i(A,B)| 
\le 3 \cdot \sum_{\alpha \in \cI \cap \cD_{i+1}} \indic{\alpha \in \Gamma_{i+1}} \le 3 Z_{1},
\]   
where~$Z_1$ is defined as in \refT{thm:UT}. 
Applying inequality~\eqref{eq:C} with $C=1$ and~$\mu =\mu_{i+1}^*$ (in the probability space conditional on $\cF_{i}$; cf.\ the beginning of \refS{sec:setupconvetions}), 
when~$\cG_{\le i}$ holds it follows that, say, 
\begin{equation}\label{Boundeachterm}
\P (Y_{i+1}\ge y_{i+1} \mid \cF_{i})\le \P(Z_{1} \ge y_{i+1}/3 \mid \cF_{i})\le
\begin{cases}
\parfrac{e\mu_{i+1}^*}{y_{i+1}/3}^{y_{i+1}/3}\le\csig^{y_{i+1}/6}   & \text{if $y_{i+1}\ge 9 \mu_{i+1}^*/\sqrt{\csig}$,}\\
1 & \text{otherwise.}
\end{cases}
\end{equation}
Comparing the definition of $\sum_{0\le i < I}\mu_{i+1}^*$ with $\mu^-$, 
using $\ctau_i \ge \tau_I \ge 1/2$ (see~\eqref{eq:taui}) and $\csig \ll 1$ we see that 
\[
\sum_{\substack{0\le i < I: \\ y_{i+1} \le 9 \mu_{i+1}^*/\sqrt{\sigma}}} y_{i+1} \le 9/\sqrt{\sigma} \cdot \sum_{0\le i < I} \mu^*_{i+1} \le 9/\sqrt{\sigma} \cdot 6\sigma\mu^- \ll \delta^2 \mu^-/9 = y .
\]
So, inserting~\eqref{Boundeachterm} into~\eqref{eq:sumbound}, 
using~\eqref{boundmu} and the definition of~$s$ 
it follows that $y/\log y = \Omega(\sqrt{n}) \gg I$ and 
\begin{equation*}
\begin{split}
\P\bigpar{Y\ge \delta^2 \mu^-/9 \text{ and } \cG_{\le I}}
&\le \sum_{\substack{(y_1,\dots,y_I)\in\mathbb{N}^I\\ \sum_{0 \le i < I}y_{i+1}=\ceil{y}}}\csig^{\ceil{y}/6 -o(y)} 
	\le (y+2)^I \cdot \csig^{y/7} \le e^{-\omega(\delta^2 \mu^-)} \le n^{-\omega(s)} ,
\end{split}
\end{equation*}
completing the proof together with the probability estimate~\eqref{eq:XUTLYmu}. 
\end{proof}

\small
\bibliographystyle{plain}

\normalsize

\begin{appendix}	
\section{Appendix}\label{sec:apx}
\begin{proof}[Proof of \refT{thm:BDI}]
We may assume that $\cI=\{1, \ldots, |\cI|\}$. 
Recalling $X= f\bigl((\xi_i)_{i \in \cI}\bigr)$, we define  
\[
D_i :=\E(X\mid \xi_1,\ldots, \xi_{i-1}, \: \xi_i=1)-\E(X\mid \xi_1,\ldots, \xi_{i-1}, \: \xi_i=0) \in [-c_i,0],
\]
where $D_i \le 0$ follows from the assumption that~$f$ is decreasing, 
and $|D_i| \le c_i$ follows, as usual, from the assumed discrete Lipschitz property of~$f$. 
Analogous to, e.g., the proof of~\cite[Theorem~1.3]{TBDI}, 
writing $p_i = \P(\xi_i=1)$ it is routine to check that 
\[
\Delta_i := \E (X\mid \xi_1,\dots,\xi_i)-\E (X\mid \xi_1,\dots,\xi_{i-1}) = D_i(1-p_i)\indic{\xi_i=1}-D_ip_i\indic{\xi_i=0}.
\]
Since $1+x \le e^x$ for $x \in \RR$ and $e^{x} \le 1+x+x^2/2$ for $x \le 0$, 
for $\theta \ge 0$ it follows easily that 
\begin{equation*}%\label{eq:MGF}
\begin{split}
\E\bigpar{e^{\theta \Delta_i}\mid \xi_1,\ldots,\xi_{i-1}}  & = (1-p_i) \cdot e^{-\theta D_ip_i} + p_i \cdot e^{\theta D_i(1-p_i)}=e^{-\theta D_ip_i}(1-p_i+p_ie^{\theta D_i})\\
& \le e^{-\theta D_ip_i + p_i(e^{\theta D_i}-1)} \le e^{\theta^2 D_i^2 p_i/2} \le  e^{\theta^2 c_i^2 p_i/2}.
\end{split}
\end{equation*}
Hence $\E\bigpar{e^{\theta \sum_{i \in \cI} \Delta_i}} \le e^{\theta^2 \lambda/2}$, where $\lambda = \sum_{i \in \cI} c_i^2p_i$. 
Noting $X-\E X = \sum_{i \in \cI} \Delta_i$, we deduce 
\[
\P(X \ge \E X + t)
=\P\bigpar{e^{\theta \sum_{i \in \cI} \Delta_i}\ge e^{\theta t}}
\le \E\bigpar{e^{\theta \sum_{i \in \cI} \Delta_i}} e^{-\theta t} \le e^{\theta^2 \lambda/2-\theta t} %= e^{-\theta t/2} 
= e^{-t^2/(2\lambda)}
\]
by choosing $\theta = t/\lambda$, completing the proof of~\eqref{eq:BDImon}. 
\end{proof}

\begin{proof}[Proof of \refL{lem:aux}]
Note that the ODE $\cPsi'(x)=e^{-\cPsi^2(x)}$ and $\cPsi(0)=0$ has the implicit solution 
\begin{equation}\label{eq:impl}
x=\int_0^{\cPsi(x)}e^{t^2}dt.
\end{equation}
For~$x \ge 0$ it follows that~$\cPsi(x)$ is strictly increasing, so that~$\cPsi'(x) \ge 0$ is strictly decreasing. 
Recalling~$\cq_i = \cPsi'(i\csig)$, we deduce $q_i \ge q_{i+1}$ and $0 \le q_i \le q_0 = 1$ for all $i \ge 0$. 

To facilitate our upcoming calculations, we first prove the auxiliary claim that, for all $i \ge 0$, 
\begin{equation}\label{eq:pii}
\cpi_i-\cPsi(i\csig) \in [\csig, 2\csig] .
\end{equation}
Indeed, using $\cPsi(0)=0$ and  monotonicity of $\cPsi'$ (for the first two inequalities) together with~$\cPsi'(0)=1$ and~$\Psi'\ge 0$ (for the last inequality) it follows that 
\[
0 \le \biggpar{\sum_{0\le j\le i-1}\csig \cPsi'(j\csig)} - \cPsi(i\csig)\le \csig (\cPsi'(0)-\cPsi'(i\csig))\le \csig ,
\]
which establishes~\eqref{eq:pii} by the definition~\eqref{def:pii} of~$\cpi_i$ and $\cPsi'(j\csig)=q_j$. 

For~\eqref{ai2}, note that by~\eqref{eq:pii} and $I=\ceil{n^{\beta}} \gg 1$ it suffices to show $\sqrt{\log x}-1\le\cPsi(x)\le \sqrt{\log x}+1$ for~$x\ge e$ (with room to spare). 
The upper bound follows from $\int_0^{\sqrt{\log x}+1}e^{t^2}dt\ge x$ and~\eqref{eq:impl}. 
Using the inequality $(y-1)e^{-2y+1}\le 1$ with $y=\sqrt{\log x}$, the lower bound follows from $\int_0^{\sqrt{\log x}-1}e^{t^2}dt\le x$ and~\eqref{eq:impl}. 

Turning to~\eqref{ai7}, note that the above calculations for~\eqref{ai2} imply $\cPsi'(x)=e^{-\cPsi^2(x)}=x^{-1+o(1)}$ as~$x \to \infty$, 
so that $\cq_I = n^{-\beta+o(1)}$. 
Together with~$q_i \ge q_I$, it then is routine to see that~\eqref{ai7} holds for~$\beta < \beta_0=1/14$. 

Now we focus on~\eqref{ai1}. 
As a warm-up, note that $\cpi_i \le \cpi_I$ for $0 \le i \le I$ by the definition~\eqref{def:pii} of~$\cpi_i$, 
and that $\cpi_I \le \sqrt{\log (I\csig)} + 2 \ll \log n = \csig^{-1/2}$ by~\eqref{ai2}, so that $\sqrt{\csig}\pi_i \le 1$. 
Next, using~\eqref{eq:pii} together with the simple inequalities $e^{-x^2}x \le 1/2$ and $e^{-x^2}x^2\le 1/2$, we also infer that 
\begin{align}
\label{eq:qipii}
\cq_i\cpi_i  &
\le e^{-\cPsi^2(i\csig)}\bigpar{\cPsi(i\csig)+2\csig} \le 1,\\
\label{eq:qipii2}
\cq_i \cpi_i^2 & 
\le e^{-\cPsi^2(i\csig)}\bigpar{\cPsi^2(i\csig) + 4 \csig\cPsi(i\csig) + 4 \csig^2} \le 1.
\end{align}
Combined with $q_i \le 1$ this implies $\cq_i \cpi_i^j \le 1$ for all $j \in \set{0,1,2}$, completing the proof of~\eqref{ai1}. 

Turning to~\eqref{ai5}, note that $\cPsi((i+1)\csig) \le \cpi_{i+1} - \csig \le \cpi_i$ by~\eqref{eq:pii}, \eqref{def:pii} and~$q_i \le 1$. 
Since~$\cPsi \ge 0$ is increasing and~$\cPsi' \ge 0$ is decreasing, 
using $q_j=\cPsi'(j\csig)$ together with $\cPsi''(x)=-2\cPsi'(x)^2\cPsi(x)$ and~\eqref{eq:qipii} it follows that 
\begin{equation}\label{eq:ai5eq}
|q_i-q_{i+1}| \le \sigma \max_{i\csig\le \xi\le (i+1)\csig} |\cPsi''(\xi)| \le \csig \cdot 2 \cPsi'(i\csig)^2 \cdot \cPsi((i+1)\csig) \le \csig \cdot 2 q_i^2 \cpi_i \le \csig \cdot 2\min\{q_i, q_i \cpi_i\} .
\end{equation}
Noting that~\eqref{eq:ai5eq} also implies~$q_i \sim q_{i+1}$, 
this completes the proof of~\eqref{ai5} since $q_i \ge q_{i+1}$. 

Finally, for~\eqref{ai4} it suffices to show $|\cq_i-\cq_{i+1}-2\csig\cq_i^2\cpi_i|\le 8\csig^2\cq_i^2$. 
Since $\cq_i = \cPsi'(i\csig)$, it follows that 
\begin{equation*}%\label{eq:ai4eq1}
\bigabs{q_i-q_{i+1}+\csig\cPsi''(i\csig)} \le \tfrac{\csig^2}{2} \max_{i\csig\le \xi\le (i+1)\csig}|\cPsi'''(\xi)| .
\end{equation*}
As~$\cPsi'(x)=e^{-\cPsi^2(x)}$, it is routine to check that $\cPsi'''(x)=2 \cPsi'(x)^3 \bigl(4\cPsi^2(x)-1\bigr)$.  
Since~$\cPsi \ge 0$ is increasing and~$\cPsi' \ge 0$ is decreasing, 
using $\cPsi((i+1)\csig) \le \cpi_i$ (as above), \eqref{eq:qipii2} and $q_i \le 1$ we infer 
\begin{equation*}%\label{eq:ai4eq2}
\max_{i\csig\le \xi\le (i+1)\csig}|\cPsi'''(\xi)| \le 2 \cPsi'(i\csig)^3 \cdot \max\bigset{4\cPsi^2((i+1)\csig), \: 1} \le 2 q_i^3\max\bigset{4\pi_{i}^2, \: 1} \le 8q_i^2 .
\end{equation*}
Furthermore, since $\cPsi''(x)=-2\cPsi'(x)^2\cPsi(x)$, using~\eqref{eq:pii} we deduce 
\begin{equation*}%\label{eq:ai4eq3}
\bigabs{\cPsi''(i\csig)  - (-2q_i^2 \cpi_i)} = \bigabs{-2 q_i^2 \cPsi(i \csig) + 2 q_i^2 \cpi_i} \le 4 \csig q_i^2  , 
\end{equation*}
which completes the proof of~\eqref{ai4}. 
\end{proof}

\end{appendix}

\end{document}